\documentclass[11pt]{article}
 
 \usepackage[utf8]{inputenc}
 \usepackage{amsthm}
 \usepackage{enumitem}
 \usepackage{hyperref}
 \usepackage{amssymb}
 \usepackage{amsmath}
 \usepackage{dsfont}
 \usepackage{mathtools}
 \usepackage{ulem}
 
 \usepackage{comment}
 
 \usepackage[left=3cm,right=3cm,top=3cm,bottom=3cm]{geometry}

 \newtheorem{thm}{Theorem}
 \newtheorem{prop}{Proposition}
 \newtheorem{lem}{Lemma}

\title{Moran model with simultaneous strong and weak selections: convergence towards a $\Lambda$-Wright-Fisher SDE}
\author{François Gaston Ged}
\date{}
\newcommand{\intervalleff}[2]{\mathopen{[}#1\,,#2\mathclose{]}}
\newcommand{\intervallefo}[2]{\mathopen{[}#1\,,#2\mathclose{)}}
\newcommand{\intervalleof}[2]{\mathopen{(}#1\,,#2\mathclose{]}}
\newcommand{\intervalleoo}[2]{\mathopen{(}#1\,,#2\mathclose{)}}

\begin{document}
\maketitle

\begin{abstract}
	We study a population model of fixed size undergoing strong selection where individuals accumulate beneficial mutations, namely the Moran model with selection.
	Schweinsberg showed that in his specific setting of \cite{S151,S152}, due to the strong selection, the genealogy of the population converges to the so-called Bolthausen-Sznitman's coalescent as the size goes to infinity.
	In this paper we sophisticate the model by splitting the population into two adversarial subgroups, that can be interpreted as two different alleles, one of which has a selective advantage over the other through a weak selection mechanism.
	We show that the proportion of disadvantaged individuals converges to the solution of a stochastic differential equation (SDE) as the population's size goes to infinity, named the $\Lambda$-Wright-Fisher SDE with selection.
	This SDE already appeared in the $\Lambda$-lookdown model with selection studied by Bah and Pardoux \cite{BP15}, in the case where the population's genealogy is described by Bolthausen-Sznitman's coalescent.\\\\
	\textbf{Subject classification:} 60J80, 92D15, 92D25, 60H10.\\
	\textbf{Keywords:} Moran model with selection, Bolthausen-Sznitman's coalescent, $\Lambda$-Wright-Fisher SDE.
\end{abstract}

\section{Introduction}

The Moran model is a classical model in population genetics.
It describes the evolution in continuous time of a haploid population with constant size, where generations are overlapping.
Every individual dies at rate $1$ and is instantaneously replaced by a copy of an individual chosen uniformly at random in the remaining population, including the individual who just died.
It is well-known that the genealogy of the Moran model is described by the so-called \textit{Kingman's coalescent}, which is the only exchangeable\footnote{\textit{Exchangeable} refers to the property that permuting the labels of several individuals in the sample leaves the law of the process unchanged.} coalescent process where merging events are only binary and non-simultaneous.
Some kind of universality of Kingman's coalescent for the genealogy of discrete time population models with fixed size was established in \cite{M00}; this result is known as M\"ohle's Lemma.
In \cite{M02}, M\"ohle also obtained convergence results towards different coalescents, and even allowed the size of the population to vary.

In the Moran model of size $N$, when the population is split into two subgroups, say the individuals carrying allele $X$ and the ones carrying allele $Y$, the proportion $(X_t)_{t\geq 0}$ of allele $X$ in the whole population converges as $N\to\infty$, when speeding up the time by a factor $N$, towards the Wright-Fisher diffusion, that is the solution to the SDE
\begin{align*}
	\mathrm{d}X_t=\frac{1}{2}\sqrt{X_t(1-X_t)}\mathrm{d}W_t,
\end{align*}
where $W$ is a standard Brownian motion.
Note the symmetry between the two alleles reflecting the fact that none of them has a selective advantage over the other.
One of the implications is the well-known duality relation between the number of blocks in Kingman's coalescent and the Wright-Fisher diffusion, as stated in Theorem 2.7 in \cite{B09}.
Namely, denoting $K_t$ the number of blocks in Kingman's coalescent at time $t$, it holds for all $x\in\intervalleoo{0}{1}$ and $k\in\mathbb{N}$ that
\begin{align}\label{Equation duality Kingman}
	\mathbb{E}(X_t^k|X_0=x)=\mathbb{E}(x^{K_t}|K_0=k).
\end{align}
The duality actually holds for more general coalescents, namely $\Lambda$-coalescents, and some Fleming-Viot processes, as shown in \cite{BL03} equation (18), that is, for $K_t$ the number of blocks at time $t$ in a $\Lambda$-coalescent and $(X_t)_{t\geq 0}$ the solution of some SDE.

The coalescence rates in a $\Lambda$-coalescent are characterized by a finite measure $\Lambda$ on $\intervalleff{0}{1}$, such that if the coalescent contains $k$ blocks at some given time, any sub-family of size $\ell$ among the $k$ blocks merge at rate given by
\begin{align*}
	\lambda_{k,\ell}:=\int_{\intervalleff{0}{1}}p^{\ell-2}(1-p)^{k-\ell}\Lambda(\mathrm{d}p).
\end{align*}
In particular, the blocks are exchangeable, in the sense that all the possible combinations of $\ell$ blocks have the same rate of merging.
Kingman's coalescent corresponds to the case $\Lambda=\delta_0$, the Dirac mass at $0$.
Another instance of $\Lambda$-coalescent is the Bolthausen-Sznitman coalescent introduced in \cite{BS98}, corresponding to $\Lambda(\mathrm{d}p)=\mathrm{d}p$.
Its importance is due to its connections to models such as spin glasses, continuous branching processes, travelling waves, some population models, see e.g. \cite{B09} and references therein.
The populations where we expect to observe the Bolthausen-Sznitman coalescent are for instance populations undergoing strong selection \cite{S152}, exploring uninhabited territories \cite{BDM07}, or quickly adapting to the environment \cite{NH13,DWF13}.
In those cases, an individual sometimes reproduces more (or faster) and generates a family of size of the same order as the population size.

\vspace*{-0.5cm}

\paragraph{Moran model with selection and $\Lambda$-lookdown model.}

When a death occurs in the Moran model, instead of choosing an individual uniformly at random to reproduce, we can include a selection component in the dynamics and choose the parent proportionally to its \textit{fitness}.
An instance of a Moran model with selection has been studied by Schweinsberg in \cite{S151} and \cite{S152}, where the individuals accumulate beneficial mutations increasing their reproduction rates; this is the model we are interested in in this work and we will describe it in more details later on.
The main result of \cite{S152} establishes that the genealogy of the Moran model with selection of Schweinsberg converges towards the Bolthausen-Sznitman coalescent, as the size of the population goes to infinity.
Let us connect it with another population model.

In \cite{BP15}, the authors study an infinite size population model called the $\Lambda$-lookdown model with selection, whose genealogy is that of the corresponding $\Lambda$-coalescent. We are of course interested in the Bolthausen-Sznitman case $\Lambda(\mathrm{d}p)=\mathrm{d}p$ on $\intervalleff{0}{1}$.
Each individual carries either allele $X$ or allele $Y$, the selection advantaging the individuals of type\footnote{In our work, \textit{type} will refer to another concept. We will get rid of the ambiguity when introducing our model.} $X$.
Theorem 3.5 in \cite{BP15} shows that the proportion of carriers of $Y$ is the solution of the following SDE:
\begin{align}\label{Bolthausen-Sznitman SDE}
	\mathcal{Y}_t=\mathcal{Y}_0-\alpha\int_0^t\mathcal{Y}_s(1-\mathcal{Y}_s)\mathrm{d}s+\int_{\intervalleff{0}{t}\times\intervalleff{0}{1}^2}p(\mathds{1}_{\left\{u\leq\mathcal{Y}_{s-}\right\}}-\mathcal{Y}_{s-})M(\mathrm{d}s,\mathrm{d}u,\mathrm{d}p),
\end{align}
where $M$ is a Poisson point process with intensity $\mathrm{d}t\otimes\mathrm{d}u\otimes\frac{\mathrm{d}p}{p^2}$, and $\alpha\geq 0$ represents the selective advantage of $X$ over $Y$.
In \cite{BP15}, Equation \eqref{Bolthausen-Sznitman SDE} is called the $\Lambda$-Wright-Fisher SDE with selection.
The previously mentioned duality \eqref{Equation duality Kingman} in this case is between the solution of \eqref{Bolthausen-Sznitman SDE} with $\alpha=0$ and the associated $\Lambda$-coalescent.
One may wonder whether it is possible to split the individuals in the Moran model with selection into two adversarial subgroups ($X$ versus $Y$), in order to observe the convergence of the proportion of the disadvantaged group $Y$ towards the solution of \eqref{Bolthausen-Sznitman SDE}.
The goal of this work is to answer this question.

\section{Model and main result}

\subsection{Previous results}

We describe more formally the Moran model with selection and the results of Schweinsberg in \cite{S151} and \cite{S152}.

We consider a population of fixed size $N\in\mathbb{N}$.
Each individual dies at rate 1, meaning that its lifetime is an exponential random variable with parameter 1.
At time $0$, the $N$ individuals carry no mutation.
Each of them acquires a mutation that adds up to its current number of mutations at rate $\mu=\mu_N$ that can depend on $N$.
We call the number of mutations carried by an individual its \textit{type}.
When a death occurs, say at time $t$, the individual is instantaneously replaced by a copy of an individual chosen in the population at time $t$, including the one who just died, independently from the past.
The parent is chosen at random proportionally to its fitness at time $t$, as explained below, and the newborn individual then inherits the type of its parent.\\
\indent For all $j\geq 0$ and $t\in\mathbb{R}_+$, we denote by $W_j(t)$ the number of individuals of type $j$ at time $t$ in the population.
The average number of mutations at time $t$ is thus given by
\begin{align*}
	M(t):=\frac{1}{N}\sum_{j\geq 0}jW_j(t).
\end{align*}
Let $s=s_N>0$ be the coefficient of selection and let the fitness of the type $j$ at time $t$ be $\max\left(1+s(j-M(t)),0\right)$.
If a death occurs at time $t$, the probability that a particular individual of type $j$ reproduces is
\begin{align*}
	F_j(t):=\frac{\max\left(1+s(j-M(t)),0\right)}{\sum_{i\geq 0}W_i(t)\max\left(1+s(i-M(t)),0\right)},
\end{align*}
which becomes
\begin{align}\label{Equation F when fitness positive}
	F_j(t)=\frac{1+s(j-M(t))}{N}
\end{align}
when all the fitnesses are positive.
We will recall later on the fact from \cite{S151} that the number of types that have appeared before time $a_NT$ is of order at most $O(k_N)$ with high probability.
Since $sk_N\to 0$ as $N\to\infty$ by Assumption $(A_3)$, the fitnesses for these types are always positive with high probability.
The neutral Moran model corresponds to the case where $s=0$, that is all the individuals have the same probability to reproduce.
We stress that thus defined, in our model, every mutation is beneficial.

Define
\begin{align}\label{definitions of k_N and a_N}
	k_N:=\frac{\log N}{\log(s/\mu)}\qquad\text{and}\qquad a_N:=\frac{\log(s/\mu)}{s},
\end{align}
which are proven in \cite{S151} to be the scaling constants such that in $a_N$ units of time, the difference between the largest type at time $t$ and the largest type at time $t+a_N$ is of order $k_N$.
\begingroup\allowdisplaybreaks
The assumptions on the parameters of the model are the following:
\begin{align*}
	(A_1)&:\hspace{0.5cm}\lim_{N\to\infty}\frac{k_N}{\log(1/s)}=\infty.\\
	(A_2)&:\hspace{0.5cm}\lim_{N\to\infty}\frac{k_N\log k_N}{\log(s/\mu)}=0.\\
	(A_3)&:\hspace{0.5cm}\lim_{N\to\infty}sk_N=0.
\end{align*}
\endgroup
In particular, it implies that $s\to 0$, $k_N,a_N\to\infty$ as $N\to\infty$, and for any $a,b>0$,
\begin{align}\label{mu et s}
	\frac{1}{N^a}\ll\mu\ll s^b.
\end{align}
We refer to \cite{S151} and \cite{S152} for more detailed discussions on these assumptions.

The main results of Schweinsberg in \cite{S151} concern the dynamics of the types distribution in the population as $N\to\infty$.
Theorem 1.4 in \cite{S151} shows that after $a_N$ units of time, the distribution of the types starts looking like that of a Gaussian variable with vanishing variance.
Theorem 1.2 in \cite{S151} states that $M(a_Nt)/k_N$ converges in probability and uniformly on compact sets of $\intervalleoo{0}{1}\cup\intervalleoo{1}{\infty}$ towards a function $m$ that we do not need to describe here.
A similar convergence holds for the difference between the fittest individuals (the highest type alive) and the mean type, as shown in Theorem 1.1 in \cite{S151}.
This fully describes the dynamics of the types distribution as $N\to\infty$ forward in time.
It enabled Scheinsberg in \cite{S152} to show, when looking backward in time, that the genealogy of the process converges in finite distributions towards the Bolthausen-Sznitman coalescent.

\paragraph{Following the fittest type.}
An important result in \cite{S152} is that, when sampling $n$ individuals in the population at time $a_NT$ and looking backwards in time, after $a_N$ units of time, all the individuals essentially share the same type with high probability and it goes on for their ancestors, the common type being the fittest (i.e. largest) type in the population.
This means that after $a_N$ units of time forward, only the individuals that were among the fittests have begotten a non-negligeable offspring.
Hence, Schweinsberg discretises the time at stopping times defined as follows: for all $j\geq 1$, let
\begin{align}\label{Equation definition tau_j}
	\tau_j:=\inf\left\{t\geq 0:W_{j-1}(t)>s/\mu\right\}.
\end{align}
In words, $\tau_j$ is approximately the time after which type $j$ mutations start occuring, making $j-1$ the fittest type in the population at time $\tau_j$; see \cite{S151} Equation (3.16) and the associated discussion.
Note that $s/\mu\to\infty$ as $N\to\infty$ by \eqref{mu et s}, but $\frac{s/\mu}{N}\to 0$.
Roughly speaking, it means that although the largest type represents a positive fraction of the whole population close to $0$, once this type reaches a size $\lceil s/\mu\rceil:=\inf\{n\geq 1:n\geq s/\mu\}$, it starts evolving in a very predictable way, which is the reason why this discretisation is powerful.
In particular, it is when a type $j$ mutation occurs relatively shortly after $\tau_j$ that a large family is likely to descend from it, due to the fact that the fitness is relative to the mean (meaning that individuals mutating faster than usual are getting strongly advantaged for reproducing).
To follow the largest type, we introduce the index
\begin{align}\label{Equation definition j(t)}
	j(t):=\sup\left\{j\geq 1:\tau_j\leq a_Nt\right\}
\end{align}
We stress that the notation $j(t)$ refers to another quantity in \cite{S151,S152}.

\subsection{Adding the weak selection dynamics}

Recall that we want to divide our population into two adversarial subgroups, say $X$ and $Y$, giving a selective advantage to $X$ such that the proportion of $Y$-individuals converges towards the solution of \eqref{Bolthausen-Sznitman SDE} as $N\to\infty$.
It is important to note that this new selection between groups $X$ and $Y$ should leave unchanged the selection between the different types.
Henceforth, we will use the name \textit{type} without further precisions to refer to the number of mutations carried by an individual, never for his group $X$ or $Y$ alone.
Nonetheless we will sometimes use the condensed \textit{type $(Y,j)$} to refer to both the group and type of an individual.

For technical reasons due to the fact that the population takes about $a_N$ units of time to reach the Bolthausen-Sznitman dynamics, we study the proportion of $Y$-individuals starting only from time $\tau_{j(2)}$, when the types' distribution already looks like a Gaussian distribution.
Let $(y_N)_{N\in\mathbb{N}}$ be a sequence in $\intervalleoo{0}{1}$ such that $y_N\to y\in\intervalleoo{0}{1}$ as $N\to\infty$ and $y_N\lceil s/\mu\rceil\in\{1,\cdots,\lceil s/\mu\rceil\}$.
At time $\tau_{j(2)}$, we mark uniformly at random exactly $y_N\lceil s/\mu\rceil$ type $j(2)-1$ individuals to be in the group $Y$.
Each individual of type $j\leq j(2)-1$ is marked with probability $y_N$.
All the others individuals in the population form the group $X$.
The usual reproduction mechanism is left unchanged by the belonging to $X$ or $Y$.
During a reproduction event, the child inherits the group of its parent.

We add a new selection, that we call \textit{weak selection}, operating between groups $X$ and $Y$ as follows.
Set the weak selection coefficient $\alpha\geq 0$, that does not depend on $N$.
Let $Y_{j}(t)$ be the number of $(Y,j)$-individuals at time $t$ for $j\geq j(2)$, and define $X_j(t)$ similarly for the $(X,j)$-individuals.
Every time a $(Y,j)$-individual acquires a $j+1$-th mutation, say at a time $t$, it is instead killed with probability
\begin{align}\label{Equation def proba of killing}
	\frac{\alpha}{q_{j+1}}\cdot\frac{X_j(t)}{X_j(t)+Y_j(t)},
\end{align}
where $q_{j+1}$ is a random variable that will be defined later on, but can be understood as the difference between the largest type and the average type in the population at the time when the type $j+1$ start appearing.
It is measurable with respect to the natural filtration of the population at this time and we suppose independence of these killings with all randomness after this time.
Nonetheless, we will see in the forthcoming Lemma \ref{Lemma bounds Delta with q} that $1/q_{j+1}$ is the natural scaling of the weak selection in order to observe a non-trivial limit as $N\to\infty$.
Each killing is immediately compensated by choosing an individual uniformly at random among the $X_j(t)$ $(X,j)$-individual to give birth to a $(X,j+1)$-individual.
Note that the dynamics of the types thus remain unchanged.

Let $(\mathcal{Y}^N_t)_{t\geq 2}$ be the c\`adl\`ag version of the process which, informally, follows the proportion of $Y$-individuals among the fittest ones, that is
\begin{align*}
	\mathcal{Y}^N_t:=\frac{Y_{j(t)-1}(\tau_{j(t)})}{\lceil s/\mu\rceil}.
\end{align*}

We now state our main result.

\begin{thm}\label{Theorem convergence SDE}
		Given that $\mathcal{Y}^N_2=y_N\in\intervalleoo{0}{1}$, for all $T>2$, the process $(\mathcal{Y}^{N}_t)_{t\in\intervalleff{2}{T}}$ converges weakly in the Skorokhod space towards the unique solution of \eqref{Bolthausen-Sznitman SDE}.
\end{thm}
The strong uniqueness of solutions of \eqref{Bolthausen-Sznitman SDE} is proven in \cite{DL12} Theorem 4.1. 

\paragraph{Organisation of the paper.}
In Section \ref{Section SDE toolbox}, we provide the technical tools for the proof of Theorem \ref{Theorem convergence SDE}.
It is divided into three subsections.

The first one recalls the notation, as well as useful results from \cite{S151} and \cite{S152}.
Proposition \ref{proposition tools} describes the evolution of the types and related quantitites, Lemma \ref{Lemma bounds Delta with q} controls the time lapse between the random discretization step $\tau_{j+1}-\tau_j$ defined in \eqref{Equation definition tau_j}.
Sometimes, an individual will have a much larger number of descendents born between these times.
Lemma \ref{Lemma convergence proportion family} approximates the law of the size of such a large family.

The second subsection introduces some notation and new populations whose dynamics need to be described.
Informally, the population identical to our model except that the most recent killings are cancelled, and the population of $(X,j)$-individuals descending from a recent killing.
They are of interest because one can retrieve the proportion of $Y$-individuals in the original population from these two.
Lemma \ref{Lemma no early killing} shows that with high probability, there is no early type $(Y,j)$ mutation that gets killed from the weak selection mechanism.

The third subsection adapts the techniques of Schweinsberg based on martingales to investigate the fluctuations of different subpopulations.
It is divided into 5 subsubsections, whose organisation is made precise at the beginning of the subsection.
Lemmas \ref{Lemma double killings negligeable} and \ref{variance Z^Y} are technical results that serve to obtain the approximation of the $(Y,j-1)$-individuals when the $(Y,j)$-individuals start appearing in the population, that is given in Lemma \ref{control Y}.
Lemma \ref{variance slow martingales} contains tools to control the numbers $Y$-individuals as well as the number of individuals descending from killings.
They will be used to derive Lemma \ref{Lemma Y after early mutation} that describes the evolution of the proportion of $Y$-individuals when one individual in the population reproduces much more that the others, due to the strong selection.
Lemma \ref{Lemma moments step proportion} shows that, in expectation, in the absence of weak selection and as long as no type $j$ individual appears too close to time $\tau_j$, the proportion of $Y$-individuals remains constant.
The expectation of the effect of the weak selection on the proportions is obtained in Lemma \ref{Lemma expectation weak selection effect}

Section \ref{Section thm SDE} is devoted to the proof of Theorem \ref{Theorem convergence SDE}.
The strategy is the following.
We first establish the tightness of $\mathcal{Y}^N$ in Lemma \ref{Lemma tightness}.
Next, we show in Lemma \ref{Lemma bound step} that the expectation of the increment of the proportion of $Y$-individuals from $j$ to $j+1$ is very close to the generator of the solution of \eqref{Bolthausen-Sznitman SDE}.
We then introduce a martingale problem in Lemma \ref{Lemma martingale problem} which states that any weak limit of $\mathcal{Y}^N$ solves it.
We conclude the argument with Lemma \ref{Lemma sol mg pb is sol SDE}, who states that this weak limit is therefore a solution of \eqref{Bolthausen-Sznitman SDE}.

\section{Toolbox}
\label{Section SDE toolbox}

\subsection{Schweinsberg's setting and notation}

In this subsection, we introduce the notation used in \cite{S151,S152} and we recall some of the results we will need.
Thus, what follows does not directly concern the dynamics of the two groups $X$ and $Y$, but rather that of the types distribution.
Set $T>2$ a positive real number, arbitrarily large.
Fix $\epsilon,\delta\in\intervalleoo{0}{1}$ such that
\begin{align}\label{delta et epsilon}
	\delta<\min\left\{\frac{1}{100},\frac{1}{19T},\epsilon^3\right\},
\end{align}
as required in \cite{S152} Equation (5.1).
We will study the process up to time $a_NT$, and control its behaviour with a probability greater than $1-\epsilon$, with accuracy $\delta$.
We shall denote by $C_i$, $i\in\mathbb{N}$, constants that can depend on $\delta,\epsilon,T$ whereas $C$ will always refer to a constant independent of those parameters, that may vary from line to line. 

We introduce some tools to study the evolution of a type.
Denote $B_j(t),D_j(t)$ respectively the birth-rate and death-rate of a particular individual of type $j$ at time $t\leq a_NT$, that is:
\begin{align}\label{defitinion of Bj, Dj and Gj}
	B_j(t)&:=(N-W_j(t))F_j(t),\\
	D_j(t)&:=\mu+1-W_j(t)F_j(t),\\
	\shortintertext{and define}
	G_j(t)&:=B_j(t)-D_j(t).
\end{align}
The value of $G_j(t)$ is the growth-rate of a particular individual of type $j$ at time $t$.
Thus, as in \cite{S151} and \cite{S152}, we can define for all $j\geq 0$
\begin{align}\label{definition of martingale Z}
	(Z_j(t))_{t\in\intervalleff{0}{a_NT}}:=\left(e^{-\int_0^tG_j(v)\mathrm{d}v}W_j(t)-\int_0^t\mu W_{j-1}(u)e^{-\int_0^uG_j(v)\mathrm{d}v}\mathrm{d}u-W_j(0)\right)_{t\in\intervalleff{0}{a_NT}}.
\end{align}
This is a square integrable martingale, the variance of which is given for $t\in\intervalleff{0}{a_NT}$ by
\begin{align}\label{variance martingale Z}
	\mathrm{Var}(Z_j(t))=\mathbb{E}\left(\int_0^te^{-2\int_0^uG_j(v)\mathrm{d}v}\Big(\mu W_{j-1}(u)+B_j(u)W_j(u)+D_j(u)W_j(u)\Big)\mathrm{d}u\right),
\end{align}
see \cite{S151} Proposition 5.1.
The role of $Z_j(t)$ is to control the fluctuations of $W_j(t)$ as follows: we rewrite \eqref{definition of martingale Z} as
\begin{align*}
	(W_j(t))_{t\in\intervalleff{0}{a_NT}}:=\left(e^{\int_0^tG_j(v)\mathrm{d}v}(W_j(0)+Z_j(t))+\int_0^t\mu W_{j-1}(u)e^{\int_u^tG_j(v)\mathrm{d}v}\mathrm{d}u\right)_{t\in\intervalleff{0}{a_NT}}.
\end{align*}
Then one sees that if $Z_j(t)$ is much smaller that $e^{\int_0^tG_j(v)\mathrm{d}v}$, then describing $W_j(t)$ reduces to describe $e^{\int_0^tG_j(v)\mathrm{d}v}$ and $W_{j-1}(u)$ up to time $t$.
To show that $Z_j$ is small with high probability, the general strategy is to bound its variance given by \eqref{variance Z^Y}. 
Roughly speaking, $Z_j$ is a martingale because $e^{\int_0^t G_j(v)\mathrm{d}v}$ is the expected number of individuals alive at time $t$ in a pure birth process starting from a single individual.
Hence, one sees that $e^{-\int_0^t G_j(v)\mathrm{d}v}W_j(t)$ would be constant in expectation, in the absence of mutations.
The integral in \ref{definition of martingale Z} is exactly the term needed to compensate these mutations and their offspring.
The variance in \eqref{variance martingale Z} follows from stochastic calculus, as shown in \cite{S151} Section 5.

We will often work with variants of the martingales $Z_j$.
We will always admit the fact that they are martingales, the reasons being the same as the one sketched above, as well as the formulas for their variances.

Let $(\mathcal{F}^N_t)_{t\geq 0}$ denote the natural filtration of $(W_j(t),j\geq 0)_{t\geq 0}$.
Using classical arguments on martingales and the fact that $(W_j(t),j\geq 0)_{t\geq 0}$ is a strong Markov process, it is shown in \cite{S151} Corollary 5.3 that if $\sigma<\kappa$ are two stopping times, then $Z_j^{\sigma,\kappa}$ defined for all $t\in\intervalleff{\sigma}{\kappa}$ by:
\begin{align*}
	Z_j^{\sigma,\kappa}(t):=e^{-\int_\sigma^{t\wedge\kappa}G_j(v)\mathrm{d}v}W_j(t\wedge\kappa)-\int_\sigma^{t\wedge\kappa}\mu W_{j-1}(u)e^{-\int_\sigma^uG_j(v)\mathrm{d}v}\mathrm{d}u-W_j(\sigma)
\end{align*}
is a squared integrable martingale with conditional variance
\begin{align*}
	&\mathrm{Var}\left(\left.Z_j^{\sigma,\kappa}(\sigma+t)\right|\mathcal{F}^N_{\sigma}\right)\\
	&\hspace{2cm}=\mathbb{E}\left(\left.\int_\sigma^{(\sigma+t)\wedge\kappa}e^{-2\int_\sigma^uG_j(v)\mathrm{d}v}\Big(\mu W_{j-1}(u)+B_j(u)W_j(u)+D_j(u)W_j(u)\Big)\mathrm{d}u\right|\mathcal{F}^N_{\sigma}\right).
\end{align*}
If $S$ is a subpopulation of individuals of type $j$ at time $\sigma$, for $t\geq\sigma$ we denote by $W_j^S(t)$ the number of individuals of type $j$ (and not above) at time $t$ such that their ancestor at time $\sigma$ is in $S$.
Define
\begin{align*}
	B_j^S(u)&:=(N-W_j^S(u))F_j(u),\\
	D_j^S(u)&:=\mu+1-W_j^S(u)F_j(u),
\end{align*}
and one can define, in the same manner as above (see \cite{S152} Corollary 4.9), the martingales $Z_j^S$, the variance of which being given by
\begin{align}\label{Equation variance martingale}
	\mathrm{Var}\left(\left.Z_j^S(\sigma+t)\right|\mathcal{F}^N_{\sigma}\right)=\mathbb{E}\left(\left.\int_\sigma^{\sigma+t}e^{-2\int_\sigma^uG_j(v)\mathrm{d}v}\Big(B_j^S(u)W_j^S(u)+D_j^S(u)W_j^S(u)\Big)\mathrm{d}u\right|\mathcal{F}^N_{\sigma}\right).
\end{align}

In \cite{S151,S152}, Schweinsberg often distinguishes whether $j$ is greater or smaller than $k_N^*:=\lceil k_N^+-1\rceil$, where
\begin{align*}
	k_N^+:=k_N+\frac{2k_N\log k_N}{\log(s/\mu)}.
\end{align*}
This constant is, roughly speaking, the first type after which the types distribution looks like a Gaussian distribution.
For $j\geq k_N^*+1$, let
\begin{align*}
	&q_j^*:=
	\begin{cases}
		j-k_N & \text{if }a_N-2a_N/k_N\leq\tau_j\leq a_N+2a_N/k_N,\\
		j-M(\tau_j) & \text{otherwise.}
	\end{cases}\\
	&q_j:=\max\{1,q_j^*\}.
\end{align*}
As we mentioned before, all the individuals have type $0$ at time $0$ and the wave dynamics starts approximately around time $a_N$, which is why the condition in the definition of $q_j^*$ is needed in \cite{S151,S152}.
However, we will be mainly interested in types $j\geq j(2)$ for which $q_j^*=j-M(\tau_j)$.
Recall that $\tau_j$ defined in \eqref{Equation definition tau_j} is approximately the time where one expects to see the first type $j$ mutations, hence $q_j$ is an approximation of the difference between $j$ and the average number of mutations when individuals of type $j$ start appearing.
Set
\begin{align}\label{definition of b}
	b:=\log\frac{24000T}{\delta^2\epsilon}.
\end{align}
For $j\geq k_N^*+1$, define
\begin{align}\label{definition xi}
	\xi_j&:=\max\left\{\tau_j,\tau_j+\frac{1}{sq_j}\log\left(\frac{1}{sq_j}\right)+\frac{b}{sq_j}\right\}.
\end{align}
We shall work on a specific event, realized with high probability, such that it holds that $\tau_j<\xi_j<\tau_{j+1}$ for all $j\geq k_N^*+1$ such that $\tau_{j+1}<a_NT$.
The goal of $\xi_j$ is to distinguish whether a mutation is faster than usual: we call a type $j$ mutation an \textit{early type $j$ mutation} if it occurs in the time interval $\intervalleff{\tau_j}{\xi_j}$.
The fitness being relative to the mean, the earlier a mutant is, the stronger is its advantage to reproduce immediately after the mutation.
The individual acquiring an early type $j$ mutation, as well as its offspring, are called \textit{early type $j$ individuals}.
In general, we will speak of early type $j$ individuals during the time interval $\intervalleff{\tau_j}{\tau_{j+1}}$ such that they still have type $j$ when thus called.
Schweinsberg showed that large families appear with the Bolthausen-Sznitman rates as a result of early mutations.

In \cite{S151,S152}, $\zeta=\zeta_N$ denotes a stopping time up to which the estimates on the evolution of the types distribution in the population hold.

Its definition requires a lot of technical considerations that are not relevant for our purposes, we therefore refer to Section 3.3 in \cite{S151} and equation (4.11) in \cite{S152} for equivalent definitions of $\zeta$.
In particular, for $N$ large enough, it holds that $\mathbb{P}(\zeta>a_NT)>1-\epsilon$.
Throughout the paper we will say that a property holds on some event $E$ if it is true for $\mathbb{P}$-almost every $\omega\in E$.
Similarly, if we say that on the event $E$, $\mathbb{E}(\beta)\leq c$ for some random variable $\beta$ and some constant $c$, we mean that $\mathbb{E}(\beta|E)\leq c$.
We shall work on $\{\zeta>a_NT\}$ so that the properties of the next proposition hold.
The results it gathers are from \cite{S151} and \cite{S152} as follows:
\begin{enumerate}[label = \textbullet]
	\item 1. is taken from both Proposition 3.3 point 1 and Proposition 3.6 point 3 in \cite{S151}.
	\item 2. is taken from Proposition 3.3 point 2 in \cite{S151}.
	\item 3. is taken from Proposition 3.3 point 3 \cite{S151}.
	\item 4. is taken from Proposition 4.4 points 1,2,3 in \cite{S152}.
	\item 5. is taken from Lemma 4.5 in \cite{S151}.
\end{enumerate}
	
\begin{prop}\label{proposition tools}
For $N$ large enough, the following hold\footnote{The listed properties hold almost surely under the conditions given in the statements, e.g. for all $j\geq k_N^*+1$ such that $\tau_{j+1}\leq\zeta\wedge a_NT$ means for all $\omega\in\{\tau_{j+1}\leq\zeta\wedge a_NT\}$.}:
\begin{enumerate}[label=\arabic*.]
	\item For all $j\geq k_N^*+1$ such that $\tau_{j+1}\leq \zeta\wedge a_NT$, no early type $j$ individual acquires a $j+1$-th mutation before time $\tau_{j+1}$.
	Furthermore, it holds that
		\begin{align*}
			\frac{a_N}{3k_N}\leq\tau_{j+1}-\tau_j\leq\frac{2a_N}{k_N},
		\end{align*}
		and on $\{\zeta>a_NT\}$, we have $\tau_{J+1}>a_NT$ for $J:=3Tk_N+k_N^*+1$, so the types greater or equal to $J+1$ have not appeared at time $a_NT$ yet.
	\item For all $j\geq k_N^*+1$ and $t\in\intervalleff{\tau_j+\frac{a_N}{4Tk_N}}{\tau_{j+1}}\cap\intervalleff{0}{\zeta\wedge a_NT}$:
		\begin{align*}
			(1-4\delta)e^{\int_{\tau_j}^tG_j(v)\mathrm{d}v}\leq\widetilde{W}_j(t)\leq (1+4\delta)e^{\int_{\tau_j}^tG_j(v)\mathrm{d}v},
		\end{align*}
		where $\widetilde{W}_j$ denote the number of non-early type $j$.
		Moreover, the upper bound holds for all $t\in\intervalleff{\xi_j}{\tau_{j+1}}\cap\intervalleff{0}{\zeta\wedge a_NT}$.
	\item For all $j\geq k_N^*+1$ and $t\in\intervalleff{\tau_{j+1}}{\tau_{\lceil j+\frac{k_N}{4}\rceil}+a_N}\cap\intervalleff{0}{\zeta\wedge a_NT}$:
		\begin{align*}
			(1-\delta)\frac{s}{\mu}e^{\int_{\tau_{j+1}}^tG_j(v)\mathrm{d}v}\leq W_j(t)\leq (1+\delta)\frac{s}{\mu}e^{\int_{\tau_{j+1}}^tG_j(v)\mathrm{d}v}.
		\end{align*}
	\item For all $j\geq k_N^*+1$ and $t\in\intervalleff{\tau_j}{\tau_{j+1}}\cap\intervalleff{0}{\zeta\wedge a_NT}$:
		\begin{align*}
			s(q_j-C_3)\leq &G_j(t)\leq s(q_j+C_3),\\
			sk_N(1-2\delta)\leq &G_j(t)\leq sk_N(e+2\delta),\\
			k_N(1-2\delta)\leq &\quad q_j\ \leq k_N(e+2\delta).
		\end{align*}
	\item For all $j\geq k_N^*+1$ such that $\tau_{j+1}\leq \zeta\wedge a_NT$, we have
		\begin{align*}
			\frac{s}{C_6\mu}\leq e^{\int_{\tau_j}^{\tau_{j+1}}G_j(v)\mathrm{d}v}\leq\frac{2s}{\mu},
		\end{align*}
	for some explicit constant $C_6$.
\end{enumerate}
\end{prop}
In \cite{S151}, Proposition 3.6 point 1 shows that on $\{\zeta>a_NT\}$, $\tau_{k_N^*+1}\leq 2a_N/k_N$ so that $\tau_{j(2)}>\tau_{k_N^*+1}$.
As explained in the introduction, it will be more convenient for us to study the process starting at time $\tau_{j(2)}$.
We also note that by point 1 of Proposition \ref{proposition tools} above and assumption $(A_3)$, on the event $\{\zeta>a_NT\}$, all the fitnesses of the individuals until time $a_NT$ are positive, i.e. \eqref{Equation F when fitness positive} holds, and therefore
\begin{align}\label{Equation G_j with positive fitnesses}
	G_j(t)=s(j-M(t))-\mu,\qquad\forall t\leq a_NT\text{ and }j\leq J,
\end{align}
where $J:=3Tk_N+k_N^*+1$ is from point 1 of the above proposition.
Moreover, for $N$ large enough, on the event $\{\zeta>a_NT\}$, for all $t\geq 0$ and every $j\leq J$, one has
\begin{align}\label{inequality Bj+Dj}
	B_j(t)+D_j(t)
	&=\frac{(N-2W_j(t))(1+s(j-M(t)))}{N}+1+\mu
	\leq 2 +sJ+\mu\nonumber\\
	&\leq 3,
\end{align}
by assumption $(A_3)$.

In \cite{S152}, the study of the process backwards in time requires to consider only types $j$'s that belong to some set $I\subset\mathbb{N}$, defined just before Lemma 6.2 in \cite{S152}.
Its definition involves a fixed parameter $t_0\in\intervalleoo{T-37}{T-2}$.
Choosing $t_0=T-3$, one gets
\begin{align}\label{Definition I}
	&I=\left\{j_1, \cdots, j_2\right\}&\nonumber\\
	\text{with}\qquad &j_1:=\max\{j:\tau_j^*\leq 2a_N\}-\lfloor 9\delta Tk_N\rfloor,\\
	&j_2:=\max\{j:\tau_j^*\leq a_N(T-1+19/k_N)\}+\lfloor 9\delta Tk_N\rfloor\nonumber,
\end{align}
where the $\tau_j^*$'s are some deterministic times, approximating the random $\tau_j$'s (see Equation (6.1) in \cite{S152}).
The relevant informations for our purposes are given by Lemma 6.2 in \cite{S152}, which shows that on the event $\{\zeta>a_NT\}$, it holds that $\tau_{j_1}<2a_N$, and $j_2\geq L+9$, where $L$ is defined in Lemma 5.1 of \cite{S152} as
\begin{align*}
	L:=\inf\{j:\tau_j\geq a_N(T-1)-3a_N/k_N\}.
\end{align*}
It entails that $\tau_{j_2}\geq \tau_L+9a_N/3k_N$ by Proposition \ref{proposition tools} point 1.
Hence, $\tau_{j_2}>a_N(T-1)$.
We thus have that on the event $\{\zeta>a_NT\}$, $j(2)\in I$ and $j(T-1)\in I$, so that for $j(2)\leq j\leq j(T-1)$, we can use the results of Schweinsberg proven for $j\in I$, since then $2a_N\leq \tau_j\leq a_N(T-1)$.
In particular, on the event $\{\zeta>a_NT\}$, the estimates in Proposition \ref{proposition tools} hold for $j\in I$ and we will thus apply the proposition for $j\in I$ without recalling that this ensures $\tau_j\leq a_NT$.

We deduce a result on the time length between $\tau_j$ and $\tau_{j+1}$ that will be useful later on.
\begin{lem}\label{Lemma bounds Delta with q}
		For all $j\in I$, conditionally given $\mathcal{F}^N_{\tau_j}$ and on the event $\{\zeta>\tau_{j+1}\}$, it holds that
		\begin{align*}
			\frac{1-2\delta}{q_j}\leq\frac{\tau_{j+1}-\tau_j}{a_N}\leq \frac{1+2\delta}{q_j}.
		\end{align*}
	\end{lem}
	
	\begin{proof}
		By Proposition \ref{proposition tools} point 4, we know that $\sup_{t\in\intervalleff{\tau_j}{\tau_{j+1}}}|G_j(t)-sq_j|\leq sC_3$.
		Equation (8.32) in \cite{S152} states that
		\begin{align*}
			(1-\delta)a_N\leq\int_{\tau_j}^{\tau_{j+1}}\frac{G_j(v)}{s}\mathrm{d}v\leq (1+\delta)a_N.
		\end{align*}
		Therefore we have that
		\begin{align*}
			\frac{1-\delta}{q_j+C_3}\leq\frac{\tau_{j+1}-\tau_j}{a_N}\leq\frac{1+\delta}{q_j-C_3},
		\end{align*}
		which implies for $N$ large enough that
		\begin{align*}
			\frac{1-2\delta}{q_j}\leq\frac{\tau_{j+1}-\tau_j}{a_N}\leq\frac{1+2\delta}{q_j},
		\end{align*}
		since by Proposition \ref{proposition tools} point 4, $q_j\geq (1-2\delta)k_N\to\infty$ as $N\to\infty$.
	\end{proof}

	We conclude the first subsection of the toolbox with a reformulation of the result of Schweinsberg in \cite{S152} showing that the law of the number of early type $j$ individuals at time $\tau_{j+1}$ can be well approximated by the rates corresponding to Bolthausen-Sznitman coalescent.
	\begin{lem}\label{Lemma convergence proportion family}
		For $N$ large enough, for all $j\in I,$ $j\geq j(2)$, conditionally given $\mathcal{F}^N_{\tau_j}$ and on the event $\{\zeta>\tau_{j+1}\}$, for any $g\in\mathcal{C}^{\infty}(\intervalleff{0}{1})$, it holds that
		\begin{align*}
			\left|q_j\int_{\intervalleoo{\epsilon}{1}}g(x)p_{S_j}(\mathrm{d}x)-\int_\epsilon^{1}g(x)\frac{\mathrm{d}x}{x^2}\right|
			\leq C(||g||_\infty + ||g'||_\infty)\epsilon,
		\end{align*}
		where $S_j$ is the proportion of early type $j$ individuals at time $\tau_{j+1}$ among the type $j$ individuals and $p_{S_j}$ its probability distribution supported on $\{0,1/\lceil s/\mu\rceil,\cdots, 1\}$.
	\end{lem}
	
	\begin{proof}
		Let $\nu(\mathrm{d}x)=\mathrm{d}x/x^2$, $x\in\intervalleof{0}{1}$.
		Lemma 7.8 in \cite{S152} shows that for all $y\in\intervalleof{\epsilon}{1-\delta}$, it holds that
		\begin{align}\label{Equation approx p_S with nu}
			\left|q_j p_{S_j}(\intervalleof{y}{1})-\nu(\intervalleoo{y}{1})\right|\leq 14\delta\nu(\intervalleof{y}{1})\leq 14\frac{\delta}{\epsilon}.
		\end{align}
		We implicitely used that the event in \cite{S152} equation (7.48) has probability going to $1$ as $N\to\infty$, see Lemmas 7.4 and 7.7 of the same paper.
		Roughly speaking, on this event the early mutants are coupled with a branching process introduced in Section 7.2 of the same paper, allowing to approximate the law of $S_j$.
		We write
		\begin{align*}
			&\left|q_j\int_{\intervalleoo{\epsilon}{1}}g(x)p_{S_j}(\mathrm{d}x)-\int_\epsilon^{1}g(x)\nu(\mathrm{d}x)\right|
			=\left|\int_{\epsilon}^{1}\mathrm{d}yg'(y)(q_jp_{S_j}-\nu)(\intervalleof{y}{1})-g(\epsilon)(q_jp_{S_j}-\nu)(\intervalleof{\epsilon}{1})\right|\\
			&\hspace{0.7cm}\leq 14\frac{\delta}{\epsilon}||g'||_\infty+||g'||_\infty\int_{1-\delta}^1\mathrm{d}y(q_jp_{S_j}+\nu)(\intervalleof{1-\delta}{1}))+14\frac{\delta}{\epsilon}||g||_\infty
			\leq C\frac{\delta}{\epsilon}(||g||_\infty+||g'||_\infty).
		\end{align*}
		We conclude using that $\delta/\epsilon<\epsilon$ by \eqref{delta et epsilon}.
	\end{proof}

\subsection{Splitting strategy to study the weak selection}

Recall that, among the type $j(2)-1$ individuals at time $\tau_{j(2)}$, we assigned $y_N\lceil s/\mu\rceil$ of them to group $Y$, and $(1-y_N)\lceil s/\mu\rceil$ to group $X$, with the weak selection mechanism explained above Theorem \ref{Theorem convergence SDE}.
One sees when $\alpha=0$ that
\begin{align*}
	\left\{W_j(t):t\leq a_NT,j\geq 0\right\}=\left\{X_j(t)+Y_j(t):t\leq a_NT,j\geq 0\right\}
\end{align*}
is exactly the model of Schweinsberg.
Moreover, when $\alpha\neq 0$, the types distribution remains unchanged (only the genealogy is altered).
To make the proofs, we will study the fluctuations of each group as if there was no weak selection and then combine it with estimates on the number of killings.
We thus introduce the notation $\check{Y}_j(t)$ for $j\in I$, $j\geq j(2)$ and $t\in\intervalleff{\tau_j}{\tau_{j+1}}$ for the number of $(Y,j)$-individuals if we had cancelled the killings of the weak selection previously described between $\intervalleff{\tau_j}{\tau_{j+1}}$, and only those ones\footnote{Note that $\check{Y}_j$ does not correspond to $Y_j$ in a population with $\alpha=0$, since the killings of type $j'\leq j$ individuals that occured before time $\tau_j$ are kept when counting $\check{Y}_j$.}.
In particular, for $t\in\intervalleff{\tau_j}{\tau_{j+1}}$, denoting $\check{X}_j(t)$ the total number of individuals at time $t$ descending from killings between $\intervalleff{\tau_j}{t}$, one can write $Y_j(t)=\check{Y}_j(t)-\check{X}_j(t)$.
Hence, our strategy is to control $\check{X}_j(t)$ and $\check{Y}_j(t)$ separately before combining them to obtain control on $Y_j(t)$.

To obtain \eqref{Bolthausen-Sznitman SDE}, the weak selection should have asymptotically no effect on the strong selection.
We define
\begin{align}\label{Equation def tau_j prime}
	\tau_j':=\tau_j+\frac{3}{sq_j}\log\frac{1}{sq_j},
\end{align}
and note that $\tau_j<\xi_j<\tau_j'<\tau_{j+1}$ by Proposition \ref{proposition tools} point 4 and Assumption $(A_3)$.
Up to this time, Schweinsberg in \cite{S151} is able to couple the early type $j$ individuals and their progeny with a branching process.
We will only need to use estimates he derived at this time $\tau_j'$.
Let $E_j$ be the event that a $Y$-individual is killed by the weak selection during an early type $j$ mutation in $\intervalleff{\tau_j}{\xi_j}$, and that the resulting $(X,j)$-individual has descendents that are alive at time $\tau_j'$.
We complete the filtration $\mathcal{F}^N$ to take into account the groups $X$ and $Y$ of the individuals.
The following lemma shows that these problematic events have negligeable probabilities to occur.

\begin{lem}\label{Lemma no early killing}
	For all $j\in I$, $j\geq j(2)$, on the event $\{\zeta>\tau_j\}$, it holds that
	\begin{align*}
		\mathbb{P}\left(E_j|\mathcal{F}^N_{\tau_j}\right)\leq\frac{C}{\epsilon q_j^2}.
	\end{align*}
\end{lem}

\begin{proof}
	By independence, the probability of $E_j$ is the product of the probabilities of a surviving early mutation and a killing, the former being upper bounded by Lemma 7.8 in \cite{S152}.
	By combining this bound and \eqref{Equation def proba of killing}, we obtain
	\begin{align*}
		\mathbb{P}\left(E_j|\mathcal{F}^N_{\tau_j},\{\zeta>\tau_j\}\right)
		&\leq C\frac{1}{\epsilon q_j}\frac{\alpha}{q_j}
		\leq \frac{C}{\epsilon q_j^2},
	\end{align*}
	as claimed.
\end{proof}

From Lemma \ref{Lemma no early killing}, we see that
\begin{align*}
	\mathbb{P}\bigg(\bigcup_{\substack{j\in I \\ j\geq j(2)}}E_j\Big|\mathcal{F}_{\tau_j}^N,\{\zeta>a_NT\}\bigg)
	\leq \sum_{\substack{j\in I \\ j\geq j(2)}}\frac{C}{\epsilon q_j^2}
	\leq \frac{CT}{\epsilon k_N}\to 0,
\end{align*}
as $N\to\infty$, where we used that the number of elements in $I$ is smaller than $J\leq 4k_NT$, see the discussion after \eqref{Definition I}.
Hence, we redefine $\zeta$ to include the first time at which an event $E_j$ occurs and with this new definition, one can still choose $N$ large enough such that $\mathbb{P}(\zeta>a_N(T-1))> 1-\epsilon$, in particular no $E_j$ occurs for any $j\in I$ with high probability.

\subsection{Expected fluctuations of the proportions}

We divide this section into 5 parts:
in the first subsection, we will look at the effect of the weak selection on type $(j-1)$, that is the second fittest type during $\intervalleff{\tau_j}{\tau_{j+1}}$ when the fittest type $j$ starts building up.
Then, in the second subsection, we will study the non-early type $(\check{Y},j)$-individuals during $\intervalleff{\xi_j}{\tau_{j+1}}$
In the third subsection we will describe the impact on the proportion of $\check{Y}$ individuals of an early mutation.
In the fourth subsection, we will introduce the discrete process idexed by $j$ following the proportion of $(Y,j)$ individuals at time $\tau_j$.
The importance of weak selection, that is the expected number of killings of $(Y,j)$-individuals occuring in $\intervalleff{\tau_j}{\tau_{j+1}}$ will be discussed in the fifth subsection.

\subsubsection{The type $j-1$ during $\intervalleff{\tau_j}{\tau_{j+1}}$}

For $t\geq \tau_j$, let $\check{X}_{j-1}^{\tau_j}(t)$ be the number of $(X,j-1)$-individuals at time $t\geq\tau_j$ descending from a killing that occured after time $\tau_j$. 
In order to properly estimate $\check{X}_j(\tau_{j+1})$, one needs to control $\check{X}_{j-1}^{\tau_j}(t)$ for all $t\in\intervalleff{\tau_j}{\tau_{j+1}}$, since type $j$ individuals can come from mutants of these type $j-1$ individuals.
The next lemma enables us to do so.

\begin{lem}\label{Lemma double killings negligeable}
	For $N$ large enough, on the event $\{\zeta>\tau_{j+1}\}$, for all $j\in I$, $j\geq j(2)$, with probability $1-o(\sqrt{\mu})$, it holds that
	\begin{align*}
		\sup_{t\in\intervalleff{\tau_j}{\tau_{j+1}}}\frac{\check{X}_{j-1}^{\tau_j}(t)}{W_{j-1}(t)}\leq\frac{3\alpha}{q_{j-1}}.
	\end{align*}
\end{lem}

\begin{proof}
	We essentially use the same techniques as in \cite{S151}.
	We admit the two following statements without proof, referring to Section 5 in \cite{S151} for details on how to prove them: For all $j\in I$, $j\geq j(2)$, we have that
	\begin{enumerate}[label = \textbullet]
		\item the process defined for $t<\tau_j$ by $Z_{j-1}^{\tau_j}(t)=0$ and for $t\geq\tau_j$ by
		\begin{align}\label{Equation def check Z}
			\check{Z}_{j-1}^{\tau_j}(t)
			:=e^{-\int_{\tau_j}^t G_{j-1}(v)\mathrm{d}v}\check{X}_{j-1}^{\tau_j}(t)-\int_{\tau_j}^t e^{-\int_{\tau_j}^u G_{j-1}(v)\mathrm{d}v}\mu\frac{\alpha}{q_{j-1}}\frac{X_{j-2}(u)Y_{j-2}(u)}{W_{j-2}(u)}\mathrm{d}u
		\end{align}
		is a mean zero, square integrable martingale;
		\item its conditional variance is given by
		\begin{align}\label{Equation cond var check Z}
			\mathrm{Var}\left(\check{Z}_{j-1}^{\tau_j}(t\wedge \zeta)|\mathcal{F}^N_{\tau_j}\right)
			&=\mathbb{E}\Bigg(\mathds{1}_{\{\zeta>\tau_j\}}\int_{\tau_j}^{t\wedge\zeta}\mathrm{d}u e^{-2\int_{\tau_j}^u G_{j-1}(v)\mathrm{d}v}\nonumber\\
			&\hspace{0.2cm}\times\bigg(\mu\frac{\alpha}{q_{j-1}}\frac{X_{j-2}(u)Y_{j-2}(u)}{W_{j-2}(u)}
			+(B_{j-1}(u)+D_{j-1}(u))\check{X}_{j-1}^{\tau_j}(u)\bigg)\bigg|\mathcal{F}^N_{\tau_j}\Bigg).
		\end{align}
	\end{enumerate}
	The first step of the proof is to bound \eqref{Equation cond var check Z}.
	Since $X_{j-2}(t)$ and $Y_{j-2}(t)$ are smaller than $W_{j-2}(t)$ by definition, Proposition \ref{proposition tools} point 3 then point 5 entail that
	\begin{align*}
		&\mathbb{E}\left(\mathds{1}_{\{\zeta>\tau_j\}}\int_{\tau_j}^{t\wedge\zeta}\mathrm{d}u e^{-2\int_{\tau_j}^u G_{j-1}(v)\mathrm{d}v}\mu\frac{\alpha}{q_{j-1}}\frac{X_{j-2}(u)Y_{j-2}(u)}{W_{j-2}(u)}\bigg|\mathcal{F}^N_{\tau_j}\right)\\
		&\hspace{3cm}\leq\mathbb{E}\left(\mathds{1}_{\{\zeta>\tau_j\}}\int_{\tau_j}^{t\wedge\zeta}\mathrm{d}u e^{-2\int_{\tau_j}^u G_{j-1}(v)\mathrm{d}v}(1+\delta)s\frac{\alpha}{q_{j-1}}e^{\int_{\tau_{j-1}}^{u}G_{j-2}(v)\mathrm{d}v}\bigg|\mathcal{F}^N_{\tau_j}\right)\\
		&\hspace{3cm}\leq\mathbb{E}\left(\mathds{1}_{\{\zeta>\tau_j\}}\int_{\tau_j}^{t\wedge\zeta}\mathrm{d}u(1+\delta)\frac{2s^2}{\mu}\frac{\alpha}{q_{j-1}}e^{-\int_{\tau_j}^u G_j(v)\mathrm{d}v}e^{-s(\tau_j-\tau_{j-1})}\bigg|\mathcal{F}^N_{\tau_j}\right),
	\end{align*}
	where we used that $G_{j-1}(v)=G_{j-2}(v)+s$ for all $v\in\intervalleff{\tau_j}{\tau_{j+1}\wedge\zeta}$ \eqref{Equation G_j with positive fitnesses}.
	Proposition \ref{proposition tools} point 4 allows us to write $\int_{\tau_j\wedge\zeta}^{t\wedge\zeta}e^{-\int_{\tau_j}^u G_j(v)\mathrm{d}v}\mathrm{d}u\leq \frac{1}{s(q_j-C_3)}(1-e^{-s(q_j-C_3)(t\wedge\zeta-\tau_j\wedge\zeta)})\leq\frac{2}{sq_j}$ since $s\to 0$ as $N\to\infty$.
	Hence, we have that
	\begin{align}\label{Equation bound in comput var check Z}
		\mathbb{E}\left(\mathds{1}_{\{\zeta>\tau_j\}}\int_{\tau_j}^{t\wedge\zeta}\mathrm{d}u e^{-2\int_{\tau_j}^u G_{j-1}(v)\mathrm{d}v}\mu\frac{\alpha}{q_{j-1}}\frac{X_{j-2}(u)Y_{j-2}(u)}{W_{j-2}(u)}\bigg|\mathcal{F}^N_{\tau_j}\right)
		\leq C\frac{s}{\mu}\frac{e^{-s(\tau_j-\tau_{j-1})}}{q_jq_{j-1}}.
	\end{align}
	On the other hand, $B_{j-1}(u)+D_{j-1}(u)\leq 3$ for all $u\geq \tau_{j-1}$ by \eqref{inequality Bj+Dj}.
	We apply Proposition \ref{proposition tools} point 4, then use \eqref{Equation def check Z} and the martingale property of $\check{Z}_{j-1}^{\tau_j}$ to write
	\begin{align*}
		&\mathbb{E}\left(\mathds{1}_{\{\zeta>\tau_j\}}\int_{\tau_j}^{t\wedge\zeta}e^{-2\int_{\tau_j}^u G_{j-1}(v)\mathrm{d}v}(B_{j-1}(u)+D_{j-1}(u))\check{X}_{j-1}^{\tau_j}(u)\mathrm{d}u\big|\mathcal{F}^N_{\tau_{j}}\right)\\
		&\hspace{1cm}\leq 3\mathbb{E}\left(\mathds{1}_{\{\zeta>\tau_j\}}\int_{\tau_j}^{t\wedge\zeta}e^{-s(q_j-C_3-1)(u-\tau_j)}e^{-\int_{\tau_j}^u G_{j-1}(v)\mathrm{d}v}\check{X}_{j-1}^{\tau_j}(u)\mathrm{d}u\big|\mathcal{F}^N_{\tau_{j}}\right)\\
		&\hspace{1cm}\leq 3\int_{\tau_j}^{t}\mathrm{d}ue^{-s(q_j-C_3-1)(u-\tau_j)}\\
		&\hspace{3cm}\times\mathbb{E}\Bigg(\check{Z}_{j-1}^{\tau_j}(u\wedge\zeta)+\mathds{1}_{\{\zeta>t\}}\int_{\tau_j}^u \mathrm{d}re^{-\int_{\tau_j}^rG_{j-1}(v)\mathrm{d}v}\mu\frac{\alpha}{q_{j-1}}\frac{X_{j-2}(r)Y_{j-2}(r)}{W_{j-2}(r)}\big|\mathcal{F}^N_{\tau_{j}}\Bigg)\\
		&\hspace{1cm}\leq 3\int_{\tau_j}^{t}\mathrm{d}ue^{-s(q_j-C_3-1)(u-\tau_j)}\mathbb{E}\left(\mathds{1}_{\{\zeta>t\}}\int_{\tau_j}^u \mathrm{d}re^{-\int_{\tau_j}^rG_{j-1}(v)\mathrm{d}v}\mu\frac{\alpha}{q_{j-1}}\frac{X_{j-2}(r)Y_{j-2}(r)}{W_{j-2}(r)}\big|\mathcal{F}^N_{\tau_{j}}\right)\\
		&\hspace{1cm}\leq \frac{C}{sq_j}\frac{s}{\mu}\frac{e^{-s(\tau_j-\tau_{j-1})}}{q_jq_{j-1}},
	\end{align*}
	where the last inequality follows from \eqref{Equation bound in comput var check Z}.
	Thus, applying Proposition \ref{proposition tools} point 5 and coming back to \eqref{Equation cond var check Z}, we have shown that
	\begin{align*}
		\mathrm{Var}\left(\check{Z}_{j-1}^{\tau_j}(t\wedge \zeta)|\mathcal{F}^N_{\tau_j}\right)
		\leq C\frac{s}{\mu}\frac{e^{-s(\tau_j-\tau_{j-1})}}{sk_N^3}.
	\end{align*}
	Proposition \ref{proposition tools} point 1 gives $s(\tau_{j}-\tau_{j-1})\geq s\frac{a_N}{3k_N}=\log(s/\mu)/3k_N\to\infty$ by assumption $(A_2)$ such that $\log(s/\mu)/3k_N\log(k_N)\to\infty$ as $N\to\infty$.
	In particular, on $\{\zeta>\tau_j\}$, it holds that
	\begin{align}\label{Equation exp(sa_N/k_N)}
		e^{-s(\tau_j-\tau_{j-1})}
		=o\left(\frac{1}{k_N}\right)
	\end{align}
	Therefore, Doob's maximal inequality for square integrable martingales yields that
	\begin{align}\label{Equation bound sup check Z}
		\mathbb{P}\left(\sup_{t\in\intervalleff{\tau_j}{\tau_{j+1}}}\frac{|\check{Z}_{j-1}^{\tau_j}(t\wedge\zeta)|}{\lceil s/\mu\rceil}>\frac{1}{k_N^2}\Big|\mathcal{F}^N_{\tau_j}\right)
		&\leq C\frac{\mu}{s^2}k_Ne^{-s(\tau_j-\tau_{j-1})}
		=o(\mu^{1/2}),
	\end{align}
	since $\mu\ll s^a$ for any $a>0$ by \eqref{mu et s}.
	Then, we use Proposition \ref{proposition tools} point 3 to write
	\begin{align*}
		\frac{\check{X}_{j-1}^{\tau_j}(t\wedge\zeta)}{W_{j-1}(t\wedge\zeta)}
		&\leq e^{-\int_{\tau_j}^{t\wedge\zeta}G_{j-1}(v)\mathrm{d}v}\frac{\check{X}_{j-1}^{\tau_j}(t\wedge\zeta)}{\lceil s/\mu\rceil}\frac{1}{1-\delta}\\
		&=\frac{1}{1-\delta}\left(\frac{\check{Z}_{j-1}^{\tau_j}(t\wedge\zeta)}{\lceil s/\mu\rceil}+\int_{\tau_j}^{t\wedge\zeta} e^{-\int_{\tau_j}^u G_{j-1}(v)\mathrm{d}v}\mu\frac{\alpha}{q_{j-1}}\frac{X_{j-2}(u)Y_{j-2}(u)}{W_{j-2}(u)\lceil s/\mu\rceil}\mathrm{d}u\right)\\
		&\leq\frac{1}{1-\delta}\left(\frac{\check{Z}_{j-1}^{\tau_j}(t\wedge\zeta)}{\lceil s/\mu\rceil}+\mu e^{\int_{\tau_{j-1}}^{\tau_j}G_{j-1}(v)\mathrm{d}v}\frac{\alpha}{q_{j-1}}\int_{\tau_j}^{t\wedge\zeta} e^{-s(u-\tau_j)}\frac{X_{j-2}(u)Y_{j-2}(u)}{W_{j-2}(u)^2(1-\delta)}\mathrm{d}u\right)\\
		&\leq\frac{1}{1-\delta}\left(\frac{\check{Z}_{j-1}^{\tau_j}(t\wedge\zeta)}{\lceil s/\mu\rceil}+2s\frac{\alpha}{q_{j-1}}\int_{\tau_j}^{t\wedge\zeta} \frac{e^{-s(u-\tau_j)}}{1-\delta}\mathrm{d}u\right)\\
		&\leq\frac{1}{1-\delta}\left(\frac{\check{Z}_{j-1}^{\tau_j}(t\wedge\zeta)}{\lceil s/\mu\rceil}+\frac{2}{1-\delta}\frac{\alpha}{q_{j-1}}\right).
	\end{align*}
	This and \eqref{Equation bound sup check Z} together conclude the proof.
\end{proof}

The following lemma will allow us to control the fluctuations of the $(Y,j+1)$-individuals by controlling that of the $(Y,j)$-individuals between $\intervalleff{\tau_j}{\tau_{j+1}}$.

\begin{lem}\label{variance Z^Y}
	For all $j\in I,$ $j\geq j(2)$, the following process is a square integrable martingale:
	\begin{align*}
		&(Z_j^Y(t))_{t\in\intervallefo{\tau_{j+1}}{\tau_{j+2}}}\\
		&\hspace{0.2cm}:=Y_j(t)e^{-\int_{\tau_{j+1}}^tG_j(v)\mathrm{d}v}-Y_j(\tau_{j+1})-\int_{\tau_{j+1}}^t\left(1-\frac{\alpha}{q_{j-1}}\frac{X_{j-1}(u)}{W_{j-1}(u)}\right)\mu Y_{j-1}(u)e^{-\int_{\tau_{j+1}}^uG_j(v)\mathrm{d}v}\mathrm{d}u.
	\end{align*}
	Moreover, for $N$ large enough, for all $t\in\intervallefo{\tau_{j+1}}{\tau_{j+2}\wedge\zeta}$, one has the following upper bound for its conditional variance:
	\begin{align*}
		\mathrm{Var}\left(Z_j^Y(t)|\mathcal{F}^N_{\tau_{j+1}}\right)\leq\frac{21}{\mu k_N}.
	\end{align*}
\end{lem}  

\begin{proof}
	Again, we admit that $Z_j^Y$ is a square integrable martingale.
	Since $\frac{\alpha}{q_{j-1}}\frac{X_{j-1}(u)}{W_{j-1}(u)}\leq\frac{\alpha}{q_{j-1}}\to 0$ as $N\to\infty$ by Proposition \ref{proposition tools} point 4, one gets the upper bound for the variance as a direct consequence of Lemma 9.27 in \cite{S151}.
\end{proof}

The next lemma shows that the evolution of $Y_{j-1}(\tau_j+t)$ until $\tau_{j+1}\wedge\zeta$ remains predictable.
\begin{lem}\label{control Y}
	For all $j\in I,$ $j\geq j(2)+1$, on the event $\{\zeta>\tau_j\}$, it holds that
	\begin{align*}
		\mathbb{P}\left(\left|Y_{j-1}(t)-Y_{j-1}(\tau_{j})e^{\int_{\tau_{j}}^tG_{j-1}(v)\mathrm{d}v}\right|\leq\delta \lceil s/\mu\rceil e^{\int_{\tau_{j}}^tG_{j-1}(v)\mathrm{d}v},\ \forall t\in\intervalleff{\tau_j}{\tau_{j+1}\wedge \zeta}\Big|\mathcal{F}^N_{\tau_{j}}\right)&\\
		=1-o\left(\frac{1}{k_N}\right)&.
	\end{align*}
	Furthermore, the same statement holds with $X_{j-1}$ instead of $Y_{j-1}$.
\end{lem}

\begin{proof}
	The statements 2,3,4 of Proposition \ref{proposition tools} hold up to time $t$ on the event $\{\zeta>t\}$.
	Write
	\begin{align}\label{Equation Y_{j-1}}
		Y_{j-1}(t)=e^{\int_{\tau_{j}}^tG_{j-1}(v)\mathrm{d}v}\left(Y_{j-1}(\tau_{j}) +\int_{\tau_{j}}^t\left(1-\frac{\alpha}{q_j}\frac{X_{j-2}(u)}{W_{j-2}(u)}\right)\mu Y_{j-2}(u)e^{-\int_{\tau_{j}}^uG_{j-1}(v)\mathrm{d}v}\mathrm{d}u+Z_{j-1}^Y(t)\right).
	\end{align}
	We shall bound the two last terms in the above parentheses.
	By point 3 of Proposition \ref{proposition tools}, one has
	\begin{align*}
		\left(1-\frac{\alpha}{q_{j-1}}\frac{X_{j-2}(u)}{W_{j-2}(u)}\right)\mu Y_{j-2}(u)e^{-\int_{\tau_{j}}^uG_{j-1}(v)\mathrm{d}v}
		&\leq (1+\delta)se^{\int_{\tau_{j-1}}^uG_{j-2}(v)\mathrm{d}v}e^{-\int_{\tau_{j}}^uG_{j-1}(v)\mathrm{d}v},\\
		\shortintertext{now remark that $G_{j-1}(v)=G_{j-2}(v)+s$ for all $v\in\intervalleff{\tau_j}{t}$ \eqref{Equation G_j with positive fitnesses} and obtain}
		&=(1+\delta)se^{\int_{\tau_{j-1}}^{\tau_{j}}G_{j-1}(v)\mathrm{d}v}e^{-s(u-\tau_{j-1})},\\
		\shortintertext{then using the point 5 of Proposition \ref{proposition tools},}
		&\leq (1+\delta)\frac{2s^2}{\mu}e^{-s(u-\tau_{j-1})}.
	\end{align*}
	Hence,
	\begin{align*}
		\int_{\tau_{j}}^t\left(1-\frac{\alpha}{q_{j-1}}\frac{X_{j-2}(u)}{W_{j-2}(u)}\right)\mu Y_{j-2}(u)e^{-\int_{\tau_{j}}^uG_{j-1}(v)\mathrm{d}v}\mathrm{d}u
		\leq (1+\delta)\frac{2s}{\mu}(e^{-s(\tau_{j}-\tau_{j-1})}-e^{-s(t-\tau_{j-1})}).
	\end{align*}
	By \eqref{Equation exp(sa_N/k_N)}, we thus have shown that for $N$ large enough, on the event $\{\zeta>t\}$,
	\begin{align}\label{Equation int 1-alpha}
		\int_{\tau_{j}}^t\left(1-\frac{\alpha}{q_{j-1}}\frac{X_{j-2}(u)}{W_{j-2}(u)}\right)\mu Y_{j-2}(u)e^{-\int_{\tau_{j}}^uG_{j-1}(v)\mathrm{d}v}\mathrm{d}u
		\leq 3\frac{s}{\mu}\times o(1/k_N).
	\end{align}
	Furthermore, using Lemma \ref{variance Z^Y} and the Doob's maximal inequality for squared integrable martingales, one has
	\begin{align*}
		\mathbb{P}\left(\sup_{t\in\intervallefo{\tau_j}{\tau_{j+1}}}\left|Z_{j-1}^Y(t\wedge \zeta)\right|>\frac{\delta}{2}\lceil s/\mu\rceil\Big|\mathcal{F}^N_{\tau_{j}}\right)
		&\leq\frac{4}{\delta^2}\mathrm{Var}\left(Z_{j-1}^Y(\tau_{j+1}\wedge\zeta)|\mathcal{F}^N_{\tau_{j}}\right)\\
		&\leq\frac{4\mu^2}{\delta^2s^2}\frac{21}{\mu k_N}
		=o\left(\frac{1}{k_N}\right),
	\end{align*}
	by \eqref{mu et s}.
	Combined with \eqref{Equation Y_{j-1}} and \eqref{Equation int 1-alpha}, this shows that conditionally given $\zeta>\tau_j$, with probability $1-o(1/k_N)$, it holds that
	\begin{align*}
		\left|Y_{j-1}(t)-Y_{j-1}(\tau_{j})e^{\int_{\tau_{j}}^tG_{j-1}(v)\mathrm{d}v}\right|\leq \delta\lceil s/\mu\rceil e^{\int_{\tau_j}^{t\wedge\zeta}G_{j-1}(v)\mathrm{d}v},\quad\forall t\in\intervalleff{\tau_j}{\tau_{j+1}\wedge \zeta},
	\end{align*}
	which completes the proof of the statement for $Y_{j-1}$.

	The proof of the statement for $X_{j-1}$ is identical.
\end{proof}

\subsubsection{The non-early individuals}

We will need the following lemma to control the non-early individuals.
	
	\begin{lem}\label{variance slow martingales}
		For $j\in I,$ $j\geq j(2)$, let $\widetilde{W}_j$ be the process which counts the number of non-early individuals, i.e. that obtain a $j$th mutation during $\intervallefo{\xi_j}{\tau_{j+1}}$ and their descendants of type $j$, and $\widetilde{Z}_j$ the associated martingale, that is for all $t\in\intervallefo{\xi_j}{\tau_{j+1}}$:
		\begin{align*}
			\widetilde{Z}_j(t):=e^{-\int_{\xi_j}^tG_j(v)\mathrm{d}v}\widetilde{W}_j(t)-\int_{\xi_j}^t\mu W_{j-1}(u)e^{-\int_{\xi_j}^uG_j(v)\mathrm{d}v}\mathrm{d}u.
		\end{align*}
		Then, its conditional variance at time $\tau_{j+1}$ satisfies
		\begin{align*}
			\mathrm{Var}\left(\mathds{1}_{\{\zeta>\tau_{j+1}\}}\widetilde{Z}_j(\tau_{j+1})\Big|\mathcal{F}^N_{\xi_j}\right)
			\leq\frac{Ce^{\int_{\tau_j}^{\xi_j}G_j(v)\mathrm{d}v}}{sk_N^2}.
		\end{align*}
		Moreover, denoting $(Y'_j(t))_{t\in\intervalleff{\xi_j}{\tau_{j+1}}}$ the process following the number of non-early $(\check{Y},j)$-individuals, the same upper bound holds for the martingale defined by
		\begin{align*}
			Z_j'(t)
			:=e^{-\int_{\xi_j}^tG_j(v)\mathrm{d}v}Y_j'(t)-\int_{\xi_j}^t\mu Y_{j-1}(u)e^{-\int_{\xi_j}^uG_j(v)\mathrm{d}v}\mathrm{d}u.
		\end{align*}
		Finally, recall that $\tau_j'$ is the stopping time defined in \eqref{Equation def tau_j prime}, for $t\geq\tau_j'$, the following process is a mean zero square integrable martingale
		\begin{align*}
			\check{Z}_j^X(t)
			:=e^{-\int_{\tau_j'}^tG_j(v)\mathrm{d}v}\check{X}_j(t)-\int_{\tau_j'}^t\mu\frac{\alpha}{q_j}\frac{X_{j-1}(u)Y_{j-1}(u)}{W_{j-1}(u)} e^{-\int_{\tau_j'}^uG_j(v)\mathrm{d}v}\mathrm{d}u-\check{X}_j(\tau_j'),
		\end{align*}
		with conditional variance
		\begin{align*}
			\mathrm{Var}\left(\mathds{1}_{\{\zeta>\tau_{j+1}\}}\check{Z}_j^X(\tau_{j+1})\Big|\mathcal{F}^N_{\tau_j'}\right)
			\leq\frac{Ce^{\int_{\tau_j}^{\tau_j'}G_j(v)\mathrm{d}v}}{sk_N^3}
		\end{align*}
	\end{lem}

	The reason for defining the martingale $\check{Z}_j^X$ only from $\tau_j'>\xi_j$ on is because $\intervalleff{\tau_j'}{\tau_{j+1}}$ contains enough of $\intervalleff{\tau_j}{\tau_{j+1}}$ while being far enough from $\xi_j$ (and a fortiori $\tau_j$) to reduce the conditional variance of $\check{Z}_j^X$ given $\mathcal{f}^N_{\xi_j}$.

	\begin{proof}
		We only show the bound of the variance of $\check{Z}_j^X$ since the other statements follow from Lemma 9.12 in \cite{S151} (the ideas of the proof therein are similar to those presented here).

		On the event $\{\zeta>\tau_j'\}$, the formula for the variance is
		\begin{align*}
			&\mathrm{Var}\left(\check{Z}_j^X(\tau_{j+1}\wedge\zeta)\Big|\mathcal{F}^N_{\tau_j'}\right)\\
			&\hspace{1cm}=\mathbb{E}\left(\int_{\tau_j'}^{\tau_{j+1}\wedge\zeta}\mathrm{d}ue^{-2\int_{\tau_j'}^uG_j(v)\mathrm{d}v}\left(\mu\frac{\alpha}{q_j}\frac{X_{j-1}(u)Y_{j-1}(u)}{W_{j-1}(u)} + (B_j(u)+D_j(u))\check{X}_j(u)\right)\Big|\mathcal{F}^N_{\tau_j'}\right)
		\end{align*}
		We focus on the first term in the parentheses.
		Since both $X_{j-1}(u)$ and $Y_{j-1}(u)$ are smaller than $W_{j-1}(u)$ by definition, Proposition \ref{proposition tools} point 3 shows that
		\begin{align*}
			&\int_{\tau_j'}^{\tau_{j+1}\wedge\zeta}\mathrm{d}ue^{-2\int_{\tau_j'}^uG_j(v)\mathrm{d}v}\mu\frac{\alpha}{q_j}\frac{X_{j-1}(u)Y_{j-1}(u)}{W_{j-1}(u)}\\
			&\hspace{4cm}\leq (1+\delta)\mu\frac{\alpha}{q_j}\lceil s/\mu\rceil e^{\int_{\tau_j}^{\tau_j'}G_{j-1}(v)\mathrm{d}v}\int_{\tau_j'}^{\tau_{j+1}\wedge\zeta}\mathrm{d}ue^{-\int_{\tau_j'}^uG_j(v)\mathrm{d}v}e^{-s(u-\tau_j')}\\
			&\hspace{4cm}\leq Cs\frac{\alpha}{q_j} e^{\int_{\tau_j}^{\tau_j'}G_{j-1}(v)\mathrm{d}v}\int_{\tau_j'}^{\tau_{j+1}\wedge\zeta}\mathrm{d}ue^{-\int_{\tau_j'}^uG_j(v)\mathrm{d}v}e^{-s(u-\tau_j')},\\
			\shortintertext{thanks to Proposition \ref{proposition tools} point 4, we then get}
			&\hspace{4cm}\leq Cs\frac{\alpha}{q_j} e^{\int_{\tau_j}^{\tau_j'}G_{j-1}(v)\mathrm{d}v}\int_{\tau_j'}^{\tau_{j+1}\wedge\zeta}\mathrm{d}ue^{-s(q_j-C_3+1)(u-\tau_j')}\\
			&\hspace{4cm}\leq C\frac{\alpha}{q_j} e^{\int_{\tau_j}^{\tau_j'}G_{j-1}(v)\mathrm{d}v}\frac{1}{q_j-C_3+1}\leq C\frac{e^{\int_{\tau_j}^{\tau_j'}G_{j-1}(v)\mathrm{d}v}}{q_j^2},
		\end{align*}
		where we used that $q_j\geq (1-2\delta) k_N\to\infty$ as $N\to\infty$ by Proposition \ref{proposition tools} point 1.

		For the other terms of the variance, by \eqref{inequality Bj+Dj}, we have
		\begin{align*}
			&\mathbb{E}\left(\int_{\tau_j'}^{\tau_{j+1}\wedge\zeta}\mathrm{d}ue^{-2\int_{\tau_j'}^uG_j(v)\mathrm{d}v}(B_j(u)+D_j(u))\check{X}_j(u)\Big|\mathcal{F}^N_{\tau_j'}\right)\\
			&\hspace{0.5cm}\leq 3\mathbb{E}\left(\int_{\tau_j'}^{\tau_{j+1}\wedge\zeta}\mathrm{d}ue^{-2\int_{\tau_j'}^uG_j(v)\mathrm{d}v}\check{X}_j(u)\Big|\mathcal{F}^N_{\tau_j'}\right)\\
			&\hspace{0.5cm}=3\mathbb{E}\left(\int_{\tau_j'}^{\tau_{j+1}\wedge\zeta}\mathrm{d}ue^{-\int_{\tau_j'}^uG_j(v)\mathrm{d}v}\left(\check{Z}_j^X(u)+\mu\frac{\alpha}{q_j}\int_{\tau_j'}^ue^{-\int_{\tau_j'}^rG_j(v)\mathrm{d}v}\frac{X_{j-1}(r)Y_{j-1}(r)}{W_{j-1}(r)}\mathrm{d}r\right)+\check{X}_j(\tau_j')\Big|\mathcal{F}^N_{\tau_j'}\right).\\
			\shortintertext{We have that $\check{X}_j(\tau_j')=0$ since $\zeta>\tau_j'$, as explained in the discussion following Lemma \ref{Lemma no early killing}. We bound the term with $X_{j-1}Y_{j-1}/W_{j-1}\leq W_{j-1}$ with Proposition \ref{proposition tools} point 3, we upper bound $-G_j(v)\leq s(q_j-C_3)$ with the Proposition \ref{proposition tools} point 4 and then use the martingale property of $\check{Z}_j^X$ to obtain}
			&\hspace{0.5cm}\leq C\mu\frac{\alpha}{q_j}\frac{X_{j-1}(\tau_j)Y_{j-1}(\tau_j)}{\lceil s/\mu\rceil}e^{\int_{\tau_j}^{\tau_j'}G_{j-1}(v)\mathrm{d}v}\mathbb{E}\left(\int_{\tau_j'}^{\tau_{j+1}\wedge\zeta}\mathrm{d}ue^{-s(q_j-C_3)(u-\tau_j')}\int_{\tau_j'}^{u}\mathrm{d}re^{-s(r-\tau_j')}\Big|\mathcal{F}^N_{\tau_j'}\right)\\
			&\hspace{0.5cm}\leq C\frac{\alpha}{q_j}e^{\int_{\tau_j}^{\tau_j'}G_{j-1}(v)\mathrm{d}v}\int_{\tau_j'}^{\infty}\mathrm{d}u\left(e^{-s(q_j-C_3)(u-\tau_j')}-e^{-s(q_j-C_3+1)(u-\tau_j')}\right)\\
			&\hspace{0.5cm}\leq \frac{Ce^{\int_{\tau_j}^{\tau_j'}G_{j-1}(v)\mathrm{d}v}}{sq_j^3},
		\end{align*}
		and the claim follows from Proposition \ref{proposition tools} point 4.
	\end{proof}

	\subsubsection{After an early mutation}
	
	We now investigate the impact of an early mutation on the proportions.
	The next lemma essentially states that the difference between the proportion of $(\check{Y},j-1)$-individuals among the type $j-1$ individuals at time $\tau_j$ to the proportion of $(Y,j+1)$ among the type $j$ individuals at time $\tau_{j+1}$ is mostly determined by the number of early type $j$ individuals.
	\begin{lem}\label{Lemma Y after early mutation}
		Let $j\in I,$ $j\geq j(2)$ and let $S$ be the proportion of early type $j$ individuals at time $\tau_{j+1}$.
		Conditionally given $\mathcal{F}^N_{\tau_j}$, on the event $\{\zeta>\tau_{j+1}\}$, the following holds with probability at least $1-C\epsilon/k_N$: If $S>0$, then the early individuals have the same ancestor at time $\tau_{j}$, and if it belongs to the group $X$, then
		\begin{align*}
			Y_{j-1}(\tau_j)(1-S)-8\delta\lceil s/\mu\rceil\leq \check{Y}_j(\tau_{j+1})\leq Y_{j-1}(\tau_j)(1-S)+8\delta\lceil s/\mu\rceil.
		\end{align*}
		Similarly, if the ancestor belongs to the group $Y$, then 
		\begin{align*}
			S\lceil s/\mu\rceil+Y_{j-1}(\tau_j)(1-S)-8\delta\lceil s/\mu\rceil\leq \check{Y}_j(\tau_{j+1})\leq S\lceil s/\mu\rceil+Y_{j-1}(\tau_{j})(1-S)+8\delta\lceil s/\mu\rceil.
		\end{align*}
	\end{lem}
	
	\begin{proof}
		Recall the notation $\widetilde{W}_j$ and $Y'_j$ for the processes following the non-early type $j$ individuals, respectively the non-early type $(\check{Y},j)$-individuals.
		We write
		\begin{align*}
			\frac{Y'_j(\tau_{j+1})}{\widetilde{W}_j(\tau_{j+1})}
			&=\frac{e^{\int_{\xi_j}^{\tau_{j+1}}G_j(v)\mathrm{d}v}\left(\int_{\xi_j}^{\tau_{j+1}}\mu Y_{j-1}(u)e^{-\int_{\xi_j}^uG_j(v)\mathrm{d}v}\mathrm{d}u+Z_j'(\tau_{j+1})\right)}{\widetilde{W}_j(\tau_{j+1})},
		\end{align*}
		where $Z_j'$ is a martingale defined in Lemma \ref{variance slow martingales}.
		Using Proposition \ref{proposition tools} point 2 and Lemma \ref{control Y}, one has with probability at least $1-o(1/k_N)$ that
		\begin{align}\label{equation Y/X}
			\frac{Y_j'(\tau_{j+1})}{\widetilde{W}_j(\tau_{j+1})}
			&\leq\frac{e^{-\int_{\tau_j}^{\xi_j}G_j(v)\mathrm{d}v}}{1-4\delta}\Bigg(\left(\frac{Y_{j-1}(\tau_j)}{\lceil s/\mu\rceil}+\delta\right)\int_{\xi_j}^{\tau_{j+1}}se^{-s(u-\xi_j)}\mathrm{d}ue^{\int_{\tau_j}^{\xi_j}G_{j-1}(v)\mathrm{d}v}+Z_j'(\tau_{j+1})\Bigg).
		\end{align}
		We compute $\int_{\xi_j}^{\tau_{j+1}}se^{-s(u-\xi_j)}\mathrm{d}u=1-e^{-s(\tau_{j+1}-\xi_j)}$, and claim that this converges to 1 as $N\to\infty$.
		Indeed, we first note that, on the event $\{\zeta>\tau_{j+1}\}$, point 4 of Proposition \ref{proposition tools} entails that
		\begin{align}\label{Equation exp(tau-xi)}
			e^{-s(\xi_j-\tau_j)}
			=(sq_j)^{1/q_j}e^{-b/q_j}
			\underset{N\to\infty}{\longrightarrow}1.
		\end{align}
		We then write
		\begin{align*}
			e^{-s(\tau_{j+1}-\xi_j)}
			&=e^{-s(\tau_{j+1}-\tau_j)}e^{s(\xi_j-\tau_j)}
			=o(1/k_N),
		\end{align*}
		where we have used \eqref{Equation exp(sa_N/k_N)}.
		Coming back to \eqref{equation Y/X}, this gives
		\begin{align}\label{eqfrac}
			\frac{Y_j'(\tau_{j+1})}{\widetilde{W}_j(\tau_{j+1})}
			&\leq\frac{1}{1-4\delta}\left(\left(\frac{Y_{j-1}(\tau_j)}{\lceil s/\mu\rceil}+\delta\right)+e^{-\int_{\tau_j}^{\xi_j}G_{j}(v)\mathrm{d}v}Z_j'(\tau_{j+1})\right).
		\end{align}
		Applying Lemma \ref{variance slow martingales} and Doob's maximal inequality for squared integrable martingales, one has
		\begin{align*}
			&\mathbb{P}\left(\mathds{1}_{\{\zeta>\tau_{j+1}\}}\left|Z_j'(\tau_{j+1})\right|>\delta e^{\int_{\tau_j}^{\xi_j}G_{j}(v)\mathrm{d}v}\bigg|\mathcal{F}^N_{\xi_j}\right)\leq C\frac{e^{-\int_{\tau_j}^{\xi_j}G_{j}(v)\mathrm{d}v}}{\delta^2sk_N^2}.
		\end{align*}
		On the event $\{\zeta>\tau_{j+1}\}$ a double application of point 4 of Proposition \ref{proposition tools} gives
		\begin{align*}
			e^{-\int_{\tau_j}^{\xi_j}G_j(v)\mathrm{d}v}
			&\leq e^{-s(q_j-C_3)(\xi_j-\tau_j)}
			\leq sq_je^{-b}(1+\delta)
			\leq C sk_Ne^{-b}.
		\end{align*}
		(Since $sk_N\to 0$ as $N\to\infty$, the constant $C_3$ has been absorbed in $C$, which does not depend on the parameters $\epsilon,\delta,T$.)
		Hence,
		\begin{align*}
			&\mathbb{P}\left(\mathds{1}_{\{\zeta>\tau_{j+1}\}}\left|Z_j'(\tau_{j+1})\right|>\delta e^{\int_{\tau_j}^{\xi_j}G_j(v)\mathrm{d}v}\bigg|\mathcal{F}^N_{\xi_j}\right)
			\leq\frac{Ce^{-b}}{\delta^2k_N}
			\leq\frac{C\epsilon}{k_N},
		\end{align*}
		where we used \eqref{delta et epsilon} and \eqref{definition of b} for the last inequality.
		This result combined with \eqref{eqfrac} entails that, on the event $\{\zeta>\tau_{j+1}\}$, with a probability greater than $1-C\epsilon/k_N$:
		\begin{align*}
			\frac{Y_j'(\tau_{j+1})}{\widetilde{W}_j(\tau_{j+1})}
			\leq\frac{1}{1-4\delta}\left(\frac{Y_{j-1}(\tau_j)}{\lceil s/\mu\rceil}+2\delta\right)
			\leq\frac{Y_{j-1}(\tau_j)}{\lceil s/\mu\rceil}+7\delta.
		\end{align*}

		For the lower bound, the same reasoning as for the upper bound gives
		\begin{align*}
			\frac{Y'_j(\tau_{j+1})}{\widetilde{W}_j(\tau_{j+1})}
			\geq\frac{1}{1+4\delta}\left(\frac{Y_{j-1}(\tau_j)}{\lceil s/\mu\rceil}-3\delta\right)
			\geq\frac{Y_{j-1}(\tau_j)}{\lceil s/\mu\rceil}-8\delta.
		\end{align*}
		Since $\widetilde{W}(\tau_{j+1})=(1-S)\lceil s/\mu\rceil$, homogenizing the bounds, we get
		\begin{equation*}
		\frac{Y_{j-1}(\tau_j)}{\lceil s/\mu\rceil}-8\delta
		\leq\frac{Y_j'(\tau_{j+1})}{(1-S)\lceil s/\mu\rceil}
		\leq\frac{Y_{j-1}(\tau_j)}{\lceil s/\mu\rceil}+8\delta
		\end{equation*}
		To conclude, conditionally given $\mathcal{F}^N_{\tau_j}$, Lemma 7.5 in \cite{S152} bounds from above the probability that two early mutations survive by $2e^{2b}/q_j^2\leq 3e^{2b}/k_N^2$.
		Then, excluding this event, if there is an early mutation in the group $X$, easy calculations lead to
		\begin{align*}
			Y_{j-1}(\tau_j)(1-S)-8\delta\lceil s/\mu\rceil\leq &\check{Y}_j(\tau_{j+1})\leq Y_{j-1}(\tau_j)(1-S)+8\delta\lceil s/\mu\rceil,\\
			S\lceil s/\mu\rceil+X_{j-1}(\tau_j)(1-S)-8\delta\lceil s/\mu\rceil\leq &X_j(\tau_{j+1})-\check{X}_j(\tau_{j+1})\\
			&\hspace{1cm}\leq S\lceil s/\mu\rceil+X_{j-1}(\tau_j)(1-S)+8\delta\lceil s/\mu\rceil.
		\end{align*}
		The cases where the early mutant is a $Y$-individual is identical, which concludes the proof.
	\end{proof}
	
\subsubsection{The discrete proportions process}

To ensure that our description of the evolution of the proportions is accurate enough, we introduce a stopped discrete time process as follows.
For all $j\in I,$ $j\geq j(2)-1$, we define
\begin{align*}
	\mathsf{Y}^N_j:=\frac{Y_{j\wedge (j(\zeta)-1)}(\tau_{(j+1)\wedge j(\zeta)})}{\lceil s/\mu\rceil},
\end{align*}
which follows the proportion of $Y$-individuals among the fittest at each time $\tau_j$ stopped at the last type $j$ before $j(\zeta)$ (recall that $j(\zeta)$ is the largest $j$ such that $\tau_j\leq a_N\zeta$).
The reason to stop the process at $j(\zeta)$ is to ensure that the results above and in particular Proposition \ref{proposition tools} apply to $\mathsf{Y}_j$.

Similarly, we denote by $\check{\mathsf{Y}}^N_j$ the process defined with $\check{Y}_{j\wedge (j(\zeta)-1)}(\tau_{(j+1)\wedge j(\zeta)})$ in place of $Y_{j\wedge (j(\zeta)-1)}(\tau_{(j+1)\wedge j(\zeta)})$, i.e. the process following the proportion of $(Y,j)$-individuals where the killings of the type $j$ mutations have been cancelled (and only those ones).
We stress that the event $\{\zeta>\tau_{j+1}\}$ is included in the above definition, in the sense that the process stops evolving at the last $\tau_j\leq\zeta$, and this fact will be kept implicit when working with $\mathsf{Y}$ or $\check{\mathsf{Y}}$.
	
We now give a lemma controlling the two first moments of the proportions' increments in the absence of weak selection, when there is no early mutation, or one that does not generate a too large family.
	
\begin{lem}\label{Lemma moments step proportion}
	Let $S_j$ be the proportion of early type $j$ individuals at time $\tau_{j+1}$ among the type $j$ individuals (potentially, $S_j$ can be $0$). 
	For all $j\in I,$ $j\geq j(2)$, it holds that
	\begin{align*}
		\mathbb{E}\left(\check{\mathsf{Y}}^N_j-\mathsf{Y}^N_{j-1}\Big|\mathcal{F}^N_{\tau_j}\right)
		&=o\left(1/k_N\right),\\
		\mathbb{E}\left(\mathds{1}_{\{S_j\leq \epsilon\}}\left(\check{\mathsf{Y}}^N_j-\mathsf{Y}^N_{j-1}\right)\Big|\mathcal{F}^N_{\tau_j}\right)
		&=o\left(1/k_N\right),\\
		\mathbb{E}\left(\mathds{1}_{\{S_j\leq\epsilon\}}\left(\check{\mathsf{Y}}^N_j-\mathsf{Y}^N_{j-1}\right)^2\Big|\mathcal{F}^N_{\tau_j}\right)
		&\leq \frac{C\epsilon}{k_N}.
	\end{align*}
\end{lem}

	\begin{proof}
	Throughout the proof, we say \textit{population} for the population of individuals for which the weak selection between the last time interval $\intervalleff{\tau_j}{\tau_{j+1}}$ has been cancelled (but for which the killings that occured before $\tau_j$ are kept).
	We will thus speak of type $(\check{Y},j)$-individuals.
	
	Fix $j\in I,$ $j\geq j(2)$.
	We first note that
	\begin{align*}
		\check{\mathsf{Y}}^N_j-\mathsf{Y}^N_{j-1}=\mathds{1}_{\{\zeta>\tau_{j+1}\}}\frac{\check{Y}_j(\tau_{j+1})-Y_{j-1}(\tau_{j})}{\lceil s/\mu\rceil}.
	\end{align*}
	We call a $Y$-individual of type $j$ at time $\tau_{j+1}$ \textit{good} if his ancestor at time $\tau_{j}$ is of type $j-1$.
	We denote by $\widehat{Y}_j(\tau_{j+1})$ the number of good $\check{Y}$ individuals at time $\tau_{j+1}$, and $K_{j}$, respectively $K_{\check{Y},j}$ the number of type $j$ individuals in the population, respectively of type $(\check{Y},j)$, at time $\tau_{j+1}$ that are not good.
	We have
	\begin{align*}
		\check{Y}_j(\tau_{j+1})
		&=K_{\check{Y},j} + \widehat{Y}_j(\tau_{j+1}).
	\end{align*}
	We pick an individual uniformly at random among the $\lceil s/\mu\rceil$ individuals of type $j$ at time $\tau_{j+1}$.
	Note that he belongs to the group of good $\check{Y}$-individuals if and only if his ancestor at time $\tau_{j}$ is in the group $Y$ with type $j-1$, and we call this event $B$.
	Let $j_{\mathrm{anc}}\in\mathbb{N}$ be the type of his ancestor at time $\tau_{j}$.
	We have in particular that $\mathbb{P}(B|\mathcal{F}^N_{\tau_j},j_{anc}=j-1, \zeta>\tau_{j+1}, S\leq\epsilon)=Y_{j-1}(\tau_j)/\lceil s/\mu\rceil$, since given this conditioning, his ancestor is independent from $S_j$ and chosen uniformly at random among the $\lceil s/\mu\rceil$ individuals of type $j-1$ at time $\tau_{j}$.
	Using that $\{\zeta>\tau_{j+1},S_j\leq\epsilon\}$ is $\mathcal{F}^N_{\tau_{j+1}}$-measurable, we then write
	\begin{align*}
		&\mathbb{E}\left(\mathds{1}_{\{\zeta>\tau_{j+1},S_j\leq\epsilon\}}\frac{\widehat{Y}_j(\tau_{j+1})}{\lceil s/\mu\rceil}\Big|\mathcal{F}^N_{\tau_j}\right)\\
		&\hspace{1cm}=\mathbb{E}\left(\mathds{1}_{\{\zeta>\tau_{j+1},S_j\leq\epsilon\}}\mathbb{P}\left(B\big|\mathcal{F}^N_{\tau_{j+1}}\right)\Big|\mathcal{F}^N_{\tau_j}\right)
		=\mathbb{P}\left(B\cap\{\zeta>\tau_{j+1},S_j\leq\epsilon\}\Big|\mathcal{F}^N_{\tau_j}\right)\\
		&\hspace{1cm}=\mathbb{P}\left(B\Big|\mathcal{F}^N_{\tau_j},j_{anc}=j-1,\zeta>\tau_{j+1},S_j\leq\epsilon\right)\mathbb{P}\left(j_{anc}=j-1,\zeta>\tau_{j+1},S_j\leq\epsilon\Big|\mathcal{F}^N_{\tau_j}\right)\\
		&\hspace{1cm}=\frac{Y_{j-1}(\tau_j)}{\lceil s/\mu\rceil}\mathbb{P}\left(j_{anc}=j-1,\zeta>\tau_{j+1},S_j\leq\epsilon\Big|\mathcal{F}^N_{\tau_j}\right).
	\end{align*}
	\begingroup\allowdisplaybreaks
	Basic properties of probability measures entail that
	\begin{align*}
		&\frac{Y_{j-1}(\tau_j)}{\lceil s/\mu\rceil}\left(\mathbb{P}(\zeta>\tau_{j+1},S_j\leq\epsilon|\mathcal{F}^N_{\tau_j})-\mathbb{P}(\zeta>\tau_{j+1},j_{\mathrm{anc}}\neq j-1|\mathcal{F}^N_{\tau_j})\right)\\
		&\hspace{4cm}\leq\mathbb{E}\left(\mathds{1}_{\{S_j\leq\epsilon\}}\frac{\widehat{Y}_j(\tau_{j+1})}{\lceil s/\mu\rceil}\Big|\mathcal{F}^N_{\tau_j}\right)
		\leq\frac{Y_{j-1}(\tau_j)}{\lceil s/\mu\rceil}\mathbb{P}(\zeta>\tau_{j+1},S_j\leq\epsilon|\mathcal{F}^N_{\tau_j}).
	\end{align*}
	Since $\mathbb{P}(\zeta>\tau_{j+1},j_{\mathrm{anc}}\neq j-1|\mathcal{F}^N_{\tau_j})=\mathbb{E}(\mathds{1}_{\{\zeta>\tau_{j+1}\}}K_j/\lceil s/\mu\rceil|\mathcal{F}^N_{\tau_j})$, we have shown that
	\begin{align}\label{Equation bound on step}
		\left|\mathbb{E}\left(\mathds{1}_{\{S_j\leq \epsilon\}}\left(\check{\mathsf{Y}}^N_j-\mathsf{Y}^N_{j-1}\right)\Big|\mathcal{F}^N_{\tau_j}\right)\right|
		&\leq \mathbb{E}\left(\mathds{1}_{\{\zeta>\tau_{j+1}\}}\frac{K_{\check{Y},j}}{\lceil s/\mu\rceil}\Big|\mathcal{F}^N_{\tau_j}\right)+\mathbb{P}\left(\zeta>\tau_{j+1},j_{\mathrm{anc}}\neq j-1|\mathcal{F}^N_{\tau_j}\right)\nonumber\\
		&\leq 2\mathbb{E}\left(\mathds{1}_{\{\zeta>\tau_{j+1}\}}\frac{K_{j}}{\lceil s/\mu\rceil}\Big|\mathcal{F}^N_{\tau_j}\right)\leq \left(\frac{\mu}{s}\right)^{1/3k_N},
	\end{align}
	where the last inequality is from \cite{S151} Lemma 6.3 and, taking the logarithm and using assumption $A_2$, one can show that the bound is $o(k_N)$.\endgroup
	\ The proof for the same bound without $\mathds{1}_{\{S_j\leq\epsilon\}}$ is identical, as this event played no particular role in the above computations.\\
	We turn our attention to the moment of order two.
	Suppose now that we independently sample two individuals, possibly the same, uniformly at random among the $\lceil s/\mu\rceil$ individuals of type $j$ at time $\tau_{j+1}$.
	Denote $j_{\mathrm{anc}}$ and $j_{\mathrm{anc}}'$ the types of their respective ancestors at time $\tau_{j}$ and let $B'$ be the event that they both belong to the good $\check{Y}$ group.
	Let $D$ be the event that the two ancestors are different with $j_{\mathrm{anc}}=j_{\mathrm{anc}}'=j-1$.
	In particular, given $D$, the ancestor of the first individual is chosen uniformly at random among the $\lceil s/\mu\rceil$ individuals of type $j-1$ at time $\tau_{j}$, and then the ancestor of the second one is chosen uniformly at random among the $\lceil s/\mu\rceil-1$ that remain, the two ancestors being independent from $S$.
	We get that
	\begin{align*}
		&\mathbb{E}\left(\mathds{1}_{\{\zeta>\tau_{j+1},S_j\leq\epsilon\}}\left(\frac{\widehat{Y}_j(\tau_{j+1})}{\lceil s/\mu\rceil}\right)^2\Big|\mathcal{F}^N_{\tau_j}\right)
		=\mathbb{P}\left(B'\cap\{\zeta>\tau_{j+1},S_j\leq\epsilon\}\Big|\mathcal{F}^N_{\tau_j}\right)\\
		&\hspace{2cm}=\mathbb{P}\left(B'\Big|\mathcal{F}^N_{\tau_j},D\cap\{\zeta>\tau_{j+1},S_j\leq\epsilon\}\right)\mathbb{P}\left(D\cap\{\zeta>\tau_{j+1},S_j\leq\epsilon\}\Big|\mathcal{F}^N_{\tau_j}\right)\\
		&\hspace{7cm}+\mathbb{P}\left(B'\cap D^c\cap\{\zeta>\tau_{j+1},S_j\leq\epsilon\}\Big|\mathcal{F}^N_{\tau_j}\right)\\
		&\hspace{2cm}\leq\left(\frac{Y_{j-1}(\tau_j)}{\lceil s/\mu\rceil}\right)^2\mathbb{P}(\zeta>\tau_{j+1},S_j\leq\epsilon|\mathcal{F}^N_{\tau_j})+\mathbb{P}\left(B'\cap D^c\cap\{\zeta>\tau_{j+1},S_j\leq\epsilon\}\Big|\mathcal{F}^N_{\tau_j}\right).
	\end{align*}
	Note that $B'\cap D^c$ is included in the event that the two sampled individuals have the same ancestor of type $j-1$ at time $\tau_{j}$.
	The probability to pick twice the same individual is $1/\lceil s/\mu\rceil$.
	The probability that two different individuals have the same ancestor on the event $\{S_j\leq\epsilon\}$ is bounded in \cite{S152}.
	More precisely, Lemmas 6.5, 6.6 and Equation (8.16) in Lemma 8.8 together show that $\mathbb{P}(B'\cap D^c\cap\{\zeta>\tau_{j+1},S_j\leq\epsilon\}|\mathcal{F}^N_{\tau_j})\leq C\epsilon/k_N$.
	We obtain that
	\begin{align*}
		&\mathbb{E}\left(\mathds{1}_{\{S_j\leq\epsilon\}}\left(\check{\mathsf{Y}}^N_j-\mathsf{Y}^N_{j-1}\right)^2\Big|\mathcal{F}^N_{\tau_j}\right)\\
		&\hspace{1cm}\leq 2\mathbb{E}\left(\mathds{1}_{\{\zeta>\tau_{j+1}\}}\left(\frac{K_j}{\lceil s/\mu\rceil}\right)^2\Big|\mathcal{F}^N_{\tau_j}\right)\\
		&\hspace{1.5cm}+2\mathbb{E}\left(\mathds{1}_{\{\zeta>\tau_{j+1},S_j\leq\epsilon\}}\frac{1}{\lceil s/\mu\rceil^2}\left(\widehat{Y}_j(\tau_{j+1})^2-2\widehat{Y}_j(\tau_{j+1})Y_{j-1}(\tau_j)+Y_{j-1}(\tau_j)^2\right)\Big|\mathcal{F}^N_{\tau_j}\right)\\
		&\hspace{1cm}\leq \frac{C\epsilon}{k_N},
	\end{align*}
	as claimed.
\end{proof}

\subsubsection{Weak selection}

For Lemma \ref{Lemma moments step proportion} to be useful, it needs to be combined with a description of the number of type $j$ individuals at time $\tau_{j+1}$ descending from killings between $\intervalleff{\tau_j}{\tau_{j+1}}$, that is $\check{X}_j(\tau_{j+1})$.
Thanks to Lemma \ref{Lemma double killings negligeable}, the expected effect of the weak selection on the proportions from $\tau_j$ to $\tau_{j+1}$ can be estimated:

\begin{lem}\label{Lemma expectation weak selection effect}
	For all $j\in I$, $j\geq j(2)$, on the event $\{\zeta>\tau_j\}$, it holds that
	\begin{align*}
		&\mathbb{E}\left(\mathds{1}_{\{\zeta>\tau_{j+1}\}}e^{-\int_{\tau_j}^{\tau_{j+1}}G_j(v)\mathrm{d}v}\check{X}_j(\tau_{j+1})\Big|\mathcal{F}^N_{\tau_j}\right)=\frac{\alpha}{q_j}\frac{X_{j-1}(\tau_j)Y_{j-1}(\tau_j)}{\lceil s/\mu\rceil^2}+O\left(\frac{\epsilon}{k_N}\right),\\
		&\mathbb{E}\left(\mathds{1}_{\{\zeta>\tau_{j+1}\}}\frac{\check{X}_j(\tau_{j+1})}{\lceil s/\mu\rceil}\Big|\mathcal{F}^N_{\tau_j}\right)=\frac{\alpha}{q_j}\frac{X_{j-1}(\tau_j)Y_{j-1}(\tau_j)}{\lceil s/\mu\rceil^2}+O\left(\frac{\epsilon}{k_N}\right),\\
		&\mathbb{E}\left(\mathds{1}_{\{\zeta>\tau_{j+1}\}}\frac{\check{X}_j(\tau_{j+1})^2}{\lceil s/\mu\rceil^2}\Big|\mathcal{F}^N_{\tau_j}\right)=o\left(\frac{1}{k_N}\right).
	\end{align*}
\end{lem}

\begin{proof}

	We address the first claim.
	Using the martingale $\check{Z}_j^X$ from Lemma \ref{variance slow martingales}, we can then combine Proposition \ref{proposition tools} point 3 and Lemma \ref{control Y} and see that
	\begin{align*}
		&\mathds{1}_{\{\zeta>\tau_{j+1}\}}e^{-\int_{\tau_j}^{\tau_{j+1}}G_j(v)\mathrm{d}v}\check{X}_j(\tau_{j+1})\\
		&\hspace{1cm}= \mathds{1}_{\{\zeta>\tau_{j+1}\}}\left(e^{-\int_{\tau_j}^{\tau_j'}G_j(v)\mathrm{d}v}\check{Z}_j^X(\tau_{j+1})+\int_{\tau_j'}^{\tau_{j+1}} e^{-\int_{\tau_j}^{u} G_j(v)\mathrm{d}v}\mu\frac{\alpha}{q_j}\frac{X_{j-1}(u)Y_{j-1}(u)}{W_{j-1}(u)}\mathrm{d}u\right)\\
		&\hspace{1cm}=\mathds{1}_{\{\zeta>\tau_{j+1}\}}\Bigg(e^{-\int_{\tau_j}^{\tau_j'}G_j(v)\mathrm{d}v}\check{Z}_j^X(\tau_{j+1})+\mu\frac{\alpha}{q_j}\left(\frac{X_{j-1}(\tau_j)Y_{j-1}(\tau_j)}{\lceil s/\mu\rceil}+O\left(\delta^2\frac{s}{\mu}\right)\right)\int_{\tau_j'}^{\tau_{j+1}} e^{-s(u-\tau_j)}\mathrm{d}u\Bigg)\\
		&\hspace{1cm}=\mathds{1}_{\{\zeta>\tau_{j+1}\}}\Bigg(e^{-\int_{\tau_j}^{\tau_j'}G_j(v)\mathrm{d}v}\check{Z}_j^X(\tau_{j+1})+\frac{\alpha}{q_j}\frac{X_{j-1}(\tau_j)Y_{j-1}(\tau_j)}{\lceil s/\mu\rceil^2}\left(e^{-s(\tau_j'-\tau_j)}-e^{-s(\tau_{j+1}-\tau_j)}\right)+O\left(\frac{\delta^2}{q_j}\right)\Bigg).
	\end{align*}
	By definition, $\tau_j'=\tau_j+\frac{3}{sq_j}\log\frac{1}{sq_j}$ hence we have that $e^{-s(\tau_j'-\tau_j)}=(sq_j)^{3/q_j}\to 1$ as $N\to\infty$ since $q_j=O(k_N)$ and Assumption $(A_1)$ gives $\log((sk_N)^{1/k_N})=(\log s + \log k_N)/k_N\to 0$.
	Thus, by \eqref{Equation exp(sa_N/k_N)}, the above becomes
	\begin{align}\label{Equation check X}
		&\mathds{1}_{\{\zeta>\tau_{j+1}\}}e^{-\int_{\tau_j}^{\tau_{j+1}}G_j(v)\mathrm{d}v}\check{X}_j(\tau_{j+1})\nonumber\\
		&\hspace{2cm}=\mathds{1}_{\{\zeta>\tau_{j+1}\}}\Bigg(e^{-\int_{\tau_j}^{\tau_j'}G_j(v)\mathrm{d}v}\check{Z}_j^X(\tau_{j+1})+\frac{\alpha}{q_j}\frac{X_{j-1}(\tau_j)Y_{j-1}(\tau_j)}{\lceil s/\mu\rceil^2}+O\left(\frac{\delta^2}{q_j}\right)\Bigg).
	\end{align}
	We use Cauchy-Schwarz inequality and Lemma \ref{variance slow martingales} to bound the expectation of the first term:
	\begin{align*}
		&\left|\mathbb{E}\left(\mathds{1}_{\{\zeta>\tau_{j+1}\}}e^{-\int_{\tau_j}^{\tau_j'}G_j(v)\mathrm{d}v}\check{Z}_j^X(\tau_{j+1})\Big|\mathcal{F}^N_{\tau_j}\right)\right|\\
		&\hspace{3cm}\leq\mathbb{E}\left(\mathds{1}_{\{\zeta>\tau_j'\}}e^{-2\int_{\tau_j}^{\tau_j'}G_j(v)\mathrm{d}v}\Big|\mathcal{F}^N_{\tau_j}\right)^{1/2}\mathbb{E}\Bigg(\mathds{1}_{\{\zeta>\tau_j'\}}\frac{e^{\int_{\tau_j}^{\tau_j'}G_j(v)\mathrm{d}v}}{sq_j^3}\Big|\mathcal{F}^N_{\tau_j}\Bigg)^{1/2}.
	\end{align*}
	By Proposition \ref{proposition tools} point 4, we have
	\begin{align*}
		\mathbb{E}\left(\mathds{1}_{\{\zeta>\tau_j'\}}e^{-2\int_{\tau_j}^{\tau_j'}G_j(v)\mathrm{d}v}\Big|\mathcal{F}^N_{\tau_j}\right)
		\leq e^{-2s(q_j-C_3)(\tau_j'-\tau_j)},
	\end{align*}
	The right-hand side above reads as $(sk_n)^6(sk_N)^{-2C_3/q_j}\leq C(sk_N)^6$ since $q_j=O(k_N)$ and by Assumption $(A_1)$, $k_N/\log(1/s)\to \infty$ as $N\to\infty$ so that $1\leq (sk_N)^{-2C_3/k_N}\leq s^{-2C_3/k_N}\to 1$ ($sk_N\to 0$ by Assumption $(A_3)$).
	Similarly, one has $e^{\int_{\tau_j}^{\tau_j'}G_j(v)\mathrm{d}v}\leq C(sq_j)^{-3}$ and we see that
	\begin{align*}
		&\left|\mathbb{E}\left(\mathds{1}_{\{\zeta>\tau_{j+1}\}}e^{-\int_{\tau_j}^{\tau_j'}G_j(v)\mathrm{d}v}\check{Z}_j^X(\tau_{j+1})\Big|\mathcal{F}^N_{\tau_j}\right)\right|
		\leq Cs
		= o\left(\frac{1}{k_N}\right),
	\end{align*}
	by Assumption $(A_3)$. Therefore taking the expectation in \eqref{Equation check X} yields
	\begin{align*}
		&\mathbb{E}\left(\mathds{1}_{\{\zeta>\tau_{j+1}\}}e^{-\int_{\tau_j}^{\tau_{j+1}}G_j(v)\mathrm{d}v}\check{X}_j(\tau_{j+1})\Big|\mathcal{F}^N_{\tau_j},\{\zeta>\tau_j\}\right)\\
		&\hspace{5cm}=\frac{\alpha}{q_j}\frac{X_{j-1}(\tau_j)Y_{j-1}(\tau_j)}{\lceil s/\mu\rceil^2}\mathbb{P}\left(\zeta>\tau_{j+1}\Big|\mathcal{F}^N_{\tau_j},\{\zeta>\tau_j\}\right)+O\left(\frac{\delta^2}{q_j}\right).
	\end{align*}
	Since $\mathbb{P}(\tau_{j+1}>\zeta)\geq\mathbb{P}(\zeta>a_NT)\geq 1-\epsilon$, the proof of the first statement of the lemma is complete.

	The second claim can be proven in a similar way by multiplying \eqref{Equation check X} by $\mathds{1}_{\{S_j=0\}}$ where $S_j$ denotes the number of early type $j$ individuals alive at time $\tau_{j+1}$. 
	Then, by Proposition \ref{proposition tools} point 2, we have on the event $\{\zeta>\tau_{j+1},S_j=0\}$ that
	\begin{align*}
		&\frac{1}{1+4\delta}e^{-\int_{\tau_j}^{\tau_{j+1}}G_j(v)\mathrm{d}v}\check{X}_j(\tau_{j+1})
		\leq\frac{\check{X}_j(\tau_{j+1})}{\lceil s/\mu\rceil}
		\leq \frac{1}{1-4\delta}e^{-\int_{\tau_j}^{\tau_{j+1}}G_j(v)\mathrm{d}v}\check{X}_j(\tau_{j+1}).
	\end{align*}
	By Lemma 7.8 in \cite{S152}, we know that $\mathbb{P}(S_j>0|\mathcal{F}^N_{\tau_j})\leq Ce^{b}/q_j$. Combined with Proposition \ref{proposition tools} point 5, we see that
	\begin{align*}
		\mathbb{E}\left(\mathds{1}_{\{\zeta>\tau_{j+1},S_j>0\}}\frac{\check{X}_j(\tau_{j+1})}{\lceil s/\mu\rceil}\Big|\mathcal{F}^N_{\tau_j}\right)
		&\leq\mathbb{E}\left(\mathds{1}_{\{\zeta>\tau_{j+1},S_j>0\}}C_6e^{\int_{\tau_j}^{\tau_{j+1}}G_j(v)\mathrm{d}v}\check{X}_j(\tau_{j+1})\Big|\mathcal{F}^N_{\tau_j}\right)\\
		&=CC_6e^b O\left(\frac{1}{k_N^2}\right),
	\end{align*}
	where the last estimate follows from multiplying \eqref{Equation check X} by $\mathds{1}_{\{S_j>0\}}$ and taking the expectation.
	
	We now show the last identity of the lemma.
	We simply note that elevating to the square then taking the expectation in \eqref{Equation check X} and using the same approximations as throughout the proof implies that
	\begin{align*}
		&\mathbb{E}\left(\mathds{1}_{\{\zeta>\tau_{j+1}\}}e^{-2\int_{\tau_j}^{\tau_{j+1}}G_j(v)\mathrm{d}v}\check{X}_j(\tau_{j+1})^2|\mathcal{F}^N_{\tau_j}\right)\\
		&\hspace{1.5cm}\leq\left(C(sq_j)^6\mathrm{Var}(\mathds{1}_{\{\zeta>\tau_{j+1}\}}\check{Z}_j(\tau_{j+1}|\mathcal{F}^N_{\tau_j}))+O\left(\frac{\alpha^2}{q_j^2}\right)\right)
		=O\left(s^2+\frac{1}{q_j^2}\right)
		=O\left(\frac{1}{k_N^2}\right),
	\end{align*}
	where we recall that $q_j\geq (1-2\delta k_N)$ by Proposition \ref{proposition tools} point 4.
	We then use Proposition \ref{proposition tools} point 5 to conclude the proof.
\end{proof}
	
	\section{Convergence towards the SDE.}
	\label{Section thm SDE}
	In this section, our strategy to show the convergence is very common when showing convergence of a Markov process to the solution of a SDE: we first establish tightness, then look at the infinitesimal generator, and show that any weak limit solves a martingale problem associated to the SDE.
	The next lemma addresses the tightness of $(\mathcal{Y}^N_t)_{t\in\intervalleff{2}{T-1}}$.
		
	\begin{lem}\label{Lemma tightness}
	Recall that $\mathcal{Y}^N_2=y_N$ with $y_N\to y\in\intervalleoo{0}{1}$.
	The sequence $\{(\mathcal{Y}^N(t))_{t\in\intervalleff{2}{T-1}};N\in\mathbb{N}\}$ is tight.
\end{lem}

\begin{proof}
	The proof uses Aldous' criterion for tightness, stated e.g. in \cite{JS03} Chapter VI Theorem 4.5.
	Let $\lambda,\theta>0$ and let $\sigma,\sigma'$ denote any two stopping times with respect to the filtration $\mathcal{F}^N$, that are bounded by $T-1$, and such that $\sigma\leq\sigma'\leq\sigma+\theta$.
	Splitting the following probability on the events $\{\zeta>a_NT\}$ and its complement entails that
	\begin{align}\label{Equation bound for tightness}
		\mathbb{P}\left(|\mathcal{Y}^N_{\sigma'}-\mathcal{Y}^N_{\sigma}|>\lambda\right)
		&\leq\mathbb{P}\left(\mathds{1}_{\{\zeta>a_NT\}}|\mathcal{Y}^N_{\sigma'}-\mathcal{Y}^N_{\sigma}|>\lambda\right)+\epsilon\nonumber \\
		&\leq\lambda^{-2}\mathbb{E}\left(\mathbb{E}\left(\mathds{1}_{\{\zeta>a_NT\}}\left(\mathcal{Y}^N_{\sigma'}-\mathcal{Y}^N_{\sigma}\right)^2\Big|\mathcal{F}^N_{\tau_{j(\sigma)}}\right)\right)+\epsilon.
	\end{align}
	We rewrite the conditional expectation as
	\begin{align}\label{Equation split Aldous}
		&\mathbb{E}\Bigg(\mathds{1}_{\{\zeta>a_NT\}}\Bigg(\sum_{j=j(\sigma)}^{j(\sigma')-1}\mathsf{Y}^N_{j}-\mathsf{Y}^N_{j-1}\Bigg)^2\Big|\mathcal{F}^N_{\tau_{j(\sigma)}}\Bigg)\nonumber\\
		&\hspace{1cm}\leq 2\mathbb{E}\Bigg(\mathds{1}_{\{\zeta>a_NT\}}\Bigg(\Bigg(\sum_{j=j(\sigma)}^{j(\sigma')-1}\check{\mathsf{Y}}^N_j-\mathsf{Y}^N_{j-1}\Bigg)^2+\Bigg(\sum_{j=j(\sigma)}^{j(\sigma')-1}\frac{\check{X}_j(\tau_{j+1})}{\lceil s/\mu\rceil}\Bigg)^2\Bigg)\Big|\mathcal{F}^N_{\tau_{j(\sigma)}}\Bigg),
	\end{align}
	where we used that $(x+y)^2\leq 2x^2+2y^2$.
	Recall that on the event $\{\zeta>a_NT\}$, the number of $j\geq 1$ such that $\sigma\leq\tau_{j}\leq\sigma'$ is at most $\lceil 3k_N\theta\rceil$ by Proposition \ref{proposition tools} point 1.
	We bound the first sum above by
	\begin{align*}
		&\mathbb{E}\Bigg(\Bigg(\sum_{j=j(\sigma)}^{j(\sigma')-1}\check{\mathsf{Y}}^N_j-\mathsf{Y}^N_{j-1}\Bigg)^2\Big|\mathcal{F}^N_{\tau_{j(\sigma)}}\Bigg)\\
		&\hspace{1.5cm}\leq 2\mathbb{E}\Bigg(\sum_{j=j(\sigma)}^{j(\sigma)+\lceil 3k_N\theta\rceil}\left(\check{\mathsf{Y}}^N_{j}-\mathsf{Y}^N_{j-1}\right)^2
		+2\sum_{\substack{j,\ell=j(\sigma)\\j<\ell}}^{j(\sigma')-1}\left(\check{\mathsf{Y}}^N_{j}-\mathsf{Y}^N_{j-1}\right)\left(\check{\mathsf{Y}}^N_{\ell}-\mathsf{Y}^N_{\ell-1}\right)\Big|\mathcal{F}^N_{\tau_{j(\sigma)}}\Bigg).
	\end{align*}	 
	We now bound the first sum of the right-hand side.
	Let $S_j$ be the number of early type $j$ individuals at time $\tau_{j+1}$.
	Using our Lemma \ref{Lemma moments step proportion} and Lemma 7.8 in \cite{S152}, we have the following bound
	\begin{align*}
		&\sum_{j=j(\sigma)}^{j(\sigma)+\lceil 3k_N\theta\rceil}\mathbb{E}\left(\mathbb{E}\left((\check{\mathsf{Y}}_j^N-\mathsf{Y}_{j-1}^N)^2\Big|\mathcal{F}^N_{\tau_j}\right)\Big|\mathcal{F}^N_{\tau_{j(\sigma)}}\right)\\
		&\hspace{2cm}\leq \sum_{j=j(\sigma)}^{j(\sigma)+\lceil 3k_N\theta\rceil}\left(\frac{C\epsilon}{k_N}+\mathbb{E}\left(\mathbb{P}\left(S_j>\epsilon,\zeta>\tau_{j+1}|\mathcal{F}^N_{\tau_j}\right)\Big|\mathcal{F}^N_{\tau_{j(\sigma)}}\right)\right)
		\leq C\epsilon\theta+\frac{\theta}{\epsilon}
		\leq C\frac{\theta}{\epsilon}.
	\end{align*}
	We next turn our attention to the double sum, we write
	\begin{align*}
		\sum_{\substack{j,\ell=j(\sigma)\\j<\ell}}^{j(\sigma')-1}\left(\check{\mathsf{Y}}^N_{j}-\mathsf{Y}^N_{j-1}\right)\left(\check{\mathsf{Y}}^N_{\ell}-\mathsf{Y}^N_{\ell-1}\right)
		\leq\sup_{n\leq \lceil 3k_N\theta\rceil}\sum_{\substack{j,\ell=j(\sigma)\\j<\ell}}^{j(\sigma)+n}\left(\check{\mathsf{Y}}^N_{j}-\mathsf{Y}^N_{j-1}\right)\left(\check{\mathsf{Y}}^N_{\ell}-\mathsf{Y}^N_{\ell-1}\right),
	\end{align*}
	therefore we have that
	\begingroup
	\allowdisplaybreaks
	\begin{align*}
		&\Bigg|\mathbb{E}\Bigg(\sum_{\substack{j,\ell=j(\sigma)\\j<\ell}}^{j(\sigma')-1}\left(\check{\mathsf{Y}}^N_{j}-\mathsf{Y}^N_{j-1}\right)\left(\check{\mathsf{Y}}^N_{\ell}-\mathsf{Y}^N_{\ell-1}\right)\Big|\mathcal{F}^N_{\tau_{j(\sigma)}}\Bigg)\Bigg|\\
		&\hspace{2cm}\leq\sup_{n\leq\lceil 3k_N\theta\rceil}\Bigg|\sum_{\substack{j,\ell=j(\sigma)\\j<\ell}}^{j(\sigma)+n}\mathbb{E}\Bigg(\left(\check{\mathsf{Y}}^N_{j}-\mathsf{Y}^N_{j-1}\right)\mathbb{E}\left(\check{\mathsf{Y}}^N_{\ell}-\mathsf{Y}^N_{\ell-1}\Big|\mathcal{F}^N_{\tau_\ell}\right)\Big|\mathcal{F}^N_{\tau_{j(\sigma)}}\Bigg)\Bigg|\\
		&\hspace{2cm}=\sum_{\substack{j,\ell=j(\sigma)\\j<\ell}}^{j(\sigma)+\lceil 3k_N\theta\rceil}\mathbb{E}\left(\left|\check{\mathsf{Y}}^N_{j}-\mathsf{Y}^N_{j-1}\right|\Big|\mathcal{F}^N_{\tau_{j(\sigma)}}\right)\times o\left(\frac{1}{k_N}\right)
		=o(1),
	\end{align*}
	\endgroup
	by Lemma \ref{Lemma moments step proportion}.

	We now bound the second sum in \eqref{Equation split Aldous}.
	Proceeding similarly as before, we obtain
	\begin{align*}
		&\mathbb{E}\Bigg(\mathds{1}_{\{\zeta>a_NT\}}\Bigg(\sum_{j=j(\sigma)}^{j(\sigma')-1}\frac{\check{X}_j(\tau_{j+1})}{\lceil s/\mu\rceil}\Bigg)^2\Big|\mathcal{F}^N_{\tau_{j(\sigma)}}\Bigg)\\
		&\hspace{0.3cm}\leq\mathbb{E}\Bigg(\mathds{1}_{\{\zeta>a_NT\}}\sum_{j=j(\sigma)}^{j(\sigma)+\lceil 3k_N\theta\rceil}\Bigg(\frac{\check{X}_j(\tau_{j+1})}{\lceil s/\mu\rceil}\Bigg)^2 + \sum_{\substack{j,\ell=j(\sigma) \\ j<\ell}}^{j(\sigma)+\lceil 3k_N\theta\rceil}\frac{\check{X}_j(\tau_{j+1})}{\lceil s/\mu\rceil}\mathbb{E}\left(\mathds{1}_{\{\zeta>a_NT\}}\frac{\check{X}_\ell(\tau_{\ell+1})}{\lceil s/\mu\rceil}\Big|\mathcal{F}^N_{\tau_\ell}\right)\Big|\mathcal{F}^N_{\tau_{j(\sigma)}}\Bigg)\\&\hspace{0.3cm}\leq \lceil 3k_N\theta\rceil o\left(\frac{1}{k_N}\right) + \lceil 3\theta k_N\rceil^2\times\frac{C}{k_N^2}
		\leq \theta C,
	\end{align*}
	by Lemma \ref{Lemma expectation weak selection effect}.
	Hence, coming back to \eqref{Equation bound for tightness}, we have shown that for all $\lambda>0$, it holds that
	\begin{align*}
		&\lim_{\theta\to 0}\limsup_{N\to\infty}\sup_{\sigma,\sigma'}\mathbb{P}\left(|\mathcal{Y}^N_{\sigma'}-\mathcal{Y}^N_\sigma|>\lambda\right)\\
		&\hspace{1.5cm}\leq \lambda^{-2}\lim_{\theta\to 0}\left(C\theta+C\frac{\theta}{\epsilon}+\epsilon\right)=\lambda^{-2}\epsilon.
	\end{align*}
	Since the left-hand side does not depend on $\epsilon$, its value is simply $0$, which shows that the sequence is tight by Aldous' criterion for tightness.
\end{proof}
	
	Define
	\begin{align*}
		\Delta_j:=\frac{\tau_{j+1}-\tau_j}{a_N}.
	\end{align*}
	Even though $\{\mathsf{Y}^N_j;j\in I,j\geq j(2)-1\}$ is not Markovian, we shall mimic a classical method for showing convergence of a Markov process through its infinitesimal generator.

	\begin{lem}\label{Lemma bound step}
		On the event $\{\zeta>\tau_{j+1}\}$, for all $f\in\mathcal{C}^\infty(\intervalleff{0}{1})$, all $j\in I$, $j\geq j(2)$ and $N$ large enough, it holds that
		\begin{align*}
			&\bigg|\frac{1}{\Delta_j}\mathbb{E}\left(f(\mathsf{Y}^N_j)-f(\mathsf{Y}^N_{j-1})\Big|\mathcal{F}^N_{\tau_j}\right)+\alpha\mathsf{Y}^N_{j-1}(1-\mathsf{Y}^N_{j-1})f'(\mathsf{Y}^N_{j-1})\\
			&\hspace{2cm}-\int_0^1 \frac{\mathrm{d}x}{x^2}\int_0^1\mathrm{d}u\Big(f\left(\mathsf{Y}^N_{j-1}(1-x)+x\mathds{1}_{\left\{u\leq\mathsf{Y}^N_{j-1}\right\}}\right)-f\left(\mathsf{Y}^N_{j-1}\right)\Big)\bigg|\\
			&\hspace{8cm}\leq C\epsilon(||f||_\infty +||f'||_\infty +||f''||_\infty).
		\end{align*}
	\end{lem}
	
	\begin{proof}
		Let $\check{\mathsf{X}}_j^N:=\check{X}_j(\tau_{j+1})/\lceil s/\mu\rceil$.
		We write
		\begin{align*}
			&\mathbb{E}\left(f(\mathsf{Y}^N_j)-f(\mathsf{Y}^N_{j-1})|\mathcal{F}^N_{\tau_j}\right)
			=\mathbb{E}\left(\left(f\left(\check{\mathsf{Y}}^N_j-\check{X}^N_j\right)-f(\mathsf{Y}^N_{j-1})\right)\Big|\mathcal{F}^N_{\tau_j}\right)\\
			&\hspace{0.2cm}=\mathbb{E}\bigg(\bigg(f\left(\check{\mathsf{Y}}^N_j-\check{\mathsf{X}}^N_j\right)-f\left(\check{\mathsf{Y}}^N_j\right)+f\left(\check{\mathsf{Y}}^N_j\right)-f(\mathsf{Y}^N_{j-1})\bigg)\Big|\mathcal{F}^N_{\tau_j}\bigg)\\
			&\hspace{0.2cm}=\mathbb{E}\bigg(\bigg(-\check{\mathsf{X}}^N_jf'\left(\check{\mathsf{Y}}^N_j\right)+O\left((\check{\mathsf{X}}^N_j)^2\right)||f''||_\infty+f\left(\check{\mathsf{Y}}^N_j\right)-f(\mathsf{Y}^N_{j-1})\bigg)\Big|\mathcal{F}^N_{\tau_j}\bigg),
		\end{align*}
		where we used the Taylor-Lagrange formula.
		The second term is 
		\begin{align*}
			\mathbb{E}\left(O((\check{X}_j)^2)|\mathcal{F}^N_{\tau_j}\right)
			=o\left(\frac{1}{k_N}\right),
		\end{align*}
		by using Lemma \ref{Lemma expectation weak selection effect}.
		We then focus on the first term.
		We have
		\begin{align*}
			\mathbb{E}\left(\check{\mathsf{X}}^N_jf'\left(\check{\mathsf{Y}}^N_j\right)\Big|\mathcal{F}^N_{\tau_j}\right)
			&=\mathbb{E}\left(\check{\mathsf{X}}^N_jf'\left(\mathsf{Y}^N_{j-1}\right)+O(\check{\mathsf{Y}}^N_j-\mathsf{Y}^N_{j-1})\Big|\mathcal{F}^N_{\tau_j}\right)\\
			&=\frac{\alpha}{q_j}\mathsf{Y}^N_{j-1}(1-\mathsf{Y}^N_{j-1})f'(\mathsf{Y}^N_{j-1})+O\left(\frac{\epsilon}{q_j}\right)+o\left(\frac{1}{k_N}\right),
		\end{align*}
		by Lemmas \ref{Lemma expectation weak selection effect} and \ref{Lemma moments step proportion}.
		We thus have shown that
		\begin{align}\label{Equation expect discrete step 2}
			\mathbb{E}\left(f(\mathsf{Y}^N_j)-f(\mathsf{Y}^N_{j-1})|\mathcal{F}^N_{\tau_j}\right)
			&=-\frac{\alpha}{q_j}\mathsf{Y}_{j-1}(1-\mathsf{Y}_{j-1})f'(\mathsf{Y}^N_{j-1})+O\left(\frac{\epsilon}{q_j}\right)\nonumber\\
			&\hspace{4cm}+\mathbb{E}\bigg(f\left(\check{\mathsf{Y}}^N_j\right)-f(\mathsf{Y}^N_{j-1})\Big|\mathcal{F}^N_{\tau_j}\bigg).
		\end{align}
		We address the last term as follows: let $S_j$ be the number of early type $j$ individuals at time $\tau_{j+1}$, we use Taylor-Lagrange Formula to get the existence of $\xi$ strictly between $\check{\mathsf{Y}}^N_j$ and $\mathsf{Y}^N_{j-1}$ such that
		\begin{align*}
			&\mathds{1}_{\{S_j<\epsilon\}}\left(f\left(\check{\mathsf{Y}}^N_j\right)-f(\mathsf{Y}^N_{j-1})\right)
			=\mathds{1}_{\{S_j<\epsilon\}}\left(\left(\check{\mathsf{Y}}^N_j-\mathsf{Y}^N_{j-1}\right)f'(\mathsf{Y}^N_{j-1})+\left(\check{\mathsf{Y}}^N_j-\mathsf{Y}^N_{j-1}\right)^2f''(\xi)\right).
		\end{align*}
		Therefore, applying Lemma \ref{Lemma moments step proportion}, we see that
		\begin{align*}
			\left|\mathbb{E}\left(\mathds{1}_{\{S_j<\epsilon\}}\left(f\left(\check{\mathsf{Y}}^N_j\right)-f(\mathsf{Y}^N_{j-1})\right)\Big|\mathcal{F}^N_{\tau_j}\right)\right|
			&\leq o\left(\frac{1}{k_N}\right)||f'||_{\infty} + O\left(\frac{\epsilon}{k_N}\right)||f''||_{\infty}
		\end{align*}
		We now turn our attention on the difference when an early mutation generates a large family.
		In the proof of Lemma \ref{Lemma moments step proportion}, we introduced the notion of "good" type $j$ individuals at time $\tau_{j+1}$, that are those whose ancestor at time $\tau_j$ is of type $j-1$.
		We denoted their number by $K_j$.
		Recall that by Proposition \ref{proposition tools} point 1, we know that the ancestors at time $\tau_j$ of these $K_j$ individuals are not of type greater or equal to $j$.
		On the other hand, Lemma 6.3 in \cite{S152} shows that $\mathbb{E}(K_j\mathds{1}_{\{\zeta>\tau_{j+1}\}}|\mathcal{F}^N_{\tau_j})\leq (s/\mu)^{1-1/3k_N}$.
		Markov's Inequality thus yields that
		\begin{align*}
			\mathbb{P}\left(\mathds{1}_{\{\zeta>\tau_{j+1}\}}\frac{K_j}{s/\mu}>\epsilon\rceil|\mathcal{F}^N_{\tau_j}\right)\leq\frac{1}{\epsilon}\left(\frac{\mu}{s}\right)^{1/3k_N}=o\left(\frac{1}{\epsilon k_N}\right),
		\end{align*}
		where the last estimate can be derived by taking the logarithm and using Assumption $(A_2)$.
		
		Let $p_{S_j}$ denote the conditional distribution of $S_j$ given $\mathcal{F}^N_{\tau_j}$, supported on $\{0,1/\lceil s/\mu\rceil,\cdots,1\}$.
		Note that if an early mutation occurs as described in Lemma \ref{Lemma Y after early mutation}, on the event $\{\tau_{j+1}<\zeta\}$ and given $\mathcal{F}^N_{\tau_j}$, the individual who generates the large family, conditionally given that it is a good type $j$ individual, is chosen uniformly at random among the $\lceil s/\mu\rceil$ type $j-1$ individuals at time $\tau_j$.
		Hence, the probability that the early individual is in group $Y$, respectively $X$, (given that there was an early type $j$ mutation) is $\mathsf{Y}^N_{j-1}$, respectively $1-\mathsf{Y}^N_{j-1}$, up to a term of order $o(1/\epsilon k_N)$, as discussed above.
		Thanks to Lemma \ref{Lemma Y after early mutation}, we can write
		\begin{align*}
			&\mathbb{E}\left(\mathds{1}_{\{\zeta>\tau_{j+1},S_j>\epsilon\}}\left(f\left(\check{\mathsf{Y}}^N_j\right)-f(\mathsf{Y}^N_{j-1})\right)\Big|\mathcal{F}^N_{\tau_j}\right)\\
			&\hspace{0.3cm}=\int_{\intervalleof{\epsilon}{1}}p_{S_j}(\mathrm{d}x)\Big(\mathsf{Y}^N_{j-1}f\left(\mathsf{Y}^N_{j-1}(1-x)+x\right)+(1-\mathsf{Y}^N_{j-1})f\left(\mathsf{Y}^N_{j-1}(1-x)\right)-f\left(\mathsf{Y}^N_{j-1}\right)\Big)\\
			&\hspace{12cm+}E+o\left(\frac{1}{\epsilon k_N}\right),
		\end{align*}
		where $E$ is the error coming from the approximation in Lemma \ref{Lemma Y after early mutation}, and the probability that this approximation does not hold.
		In particular, we have that
		\begin{align*}
			|E|&\leq p_{S_j}(\intervalleof{\epsilon}{1})\bigg(\sup_{-8\delta<z<8\delta}\left|f\left(\mathsf{Y}^N_{j-1}(1-x)+x+z\right)-f\left(\mathsf{Y}^N_{j-1}(1-x)+x\right)\right|\\
			&\hspace{3cm}+\sup_{-8\delta<z<8\delta}\left|f\left(\mathsf{Y}^N_{j-1}(1-x)+z\right)-f\left(\mathsf{Y}^N_{j-1}(1-x)\right)\right|\bigg)+\frac{C\epsilon}{k_N}||f||_{\infty}\\
			&\leq \frac{1+13\delta}{\epsilon q_j}16\delta||f'||_{\infty}+\frac{C\epsilon}{k_N}||f||_{\infty}
			\leq C\frac{\epsilon}{q_j}(||f||_\infty+||f'||_\infty),
		\end{align*}
		where we used Lemma 7.8 in \cite{S152}, Proposition \ref{proposition tools} point 4, and that $\epsilon>\delta^3$ by \eqref{delta et epsilon}.
		Lemma \ref{Lemma convergence proportion family} allows us to write
		\begin{align*}
			&\bigg|\mathbb{E}\left(\mathds{1}_{\{\zeta>\tau_{j+1},S_j>\epsilon\}}\left(f\left(\check{\mathsf{Y}}^N_k\right)-f(\mathsf{Y}^N_{j-1})\right)\Big|\mathcal{F}^N_{\tau_j}\right)\\
			&\hspace{1cm}-\frac{1}{q_j}\int_\epsilon^1 \frac{\mathrm{d}x}{x^2}\int_0^1\mathrm{d}u\Big(f\left(\mathsf{Y}^N_{j-1}(1-x)+x\mathds{1}_{\left\{u\leq\mathsf{Y}^N_{j-1}\right\}}\right)-f\left(\mathsf{Y}^N_{j-1}\right)\Big)\bigg|\\
			&\hspace{8cm}\leq C\frac{\epsilon}{q_j}\left(||f'||_\infty+||f||_\infty\right).
		\end{align*}
		We leave to the reader the proof of the following bound:
		\begin{align*}
			\left|\int_0^\epsilon \frac{\mathrm{d}x}{x^2}\int_0^1\mathrm{d}u\Big(f\left(\mathsf{Y}^N_{j-1}(1-x)+x\mathds{1}_{\left\{u\leq\mathsf{Y}^N_{j-1}\right\}}\right)-f\left(\mathsf{Y}^N_{j-1}\right)\Big)\right|
			\leq\epsilon ||f''||_{\infty}
		\end{align*}
		We thus have shown that
		\begin{align*}
			&\bigg|q_j\mathbb{E}\left(f(\mathsf{Y}^N_j)-f(\mathsf{Y}^N_{j-1})\Big|\mathcal{F}^N_{\tau_j}\right)+\alpha\mathsf{Y}^N_{j-1}(1-\mathsf{Y}^N_{j-1})f'(\mathsf{Y}^N_{j-1})\\
			&\hspace{2cm}-\int_0^1 \frac{\mathrm{d}x}{x^2}\int_0^1\mathrm{d}u\Big(f\left(\mathsf{Y}^N_{j-1}(1-x)+x\mathds{1}_{\left\{u\leq\mathsf{Y}^N_{j-1}\right\}}\right)-f\left(\mathsf{Y}^N_{j-1}\right)\Big)\bigg|\\
			&\hspace{9cm}\leq C\epsilon(||f||_\infty+||f'||_\infty+||f''||_\infty).
		\end{align*}
		On the event $\{\zeta>\tau_{j+1}\}$, thanks to Lemma \ref{Lemma bounds Delta with q}, we can replace $q_j$ by $1/\Delta_j$, modifying slightly the constant $C$, which nonetheless would not depend on $\epsilon,\delta$ and $T$.
	\end{proof}
	
	\begin{lem}\label{Lemma martingale problem}
		Any weak limit $(\mathcal{Y}_t)_{t\in\intervalleff{2}{T-1}}$ of $(\mathcal{Y}^N_t)_{t\in\intervalleff{2}{T-1}}$ solves the following martingale problem
		\begin{align*}
			M_t&=f(\mathcal{Y}_t)-f(y)-\int_2^t-\alpha\mathcal{Y}_v(1-\mathcal{Y}_v)f'(\mathcal{Y}_v)\mathrm{d}v\\
			&\hspace{0.7cm}-\int_2^t\mathrm{d}v\int_{\intervalleff{0}{1}^2}\mathrm{d}u\frac{\mathrm{d}p}{p^2}\Big(f\left(\mathcal{Y}_{v}+p\left(\mathds{1}_{\{u\leq\mathcal{Y}_{v}\}}-\mathcal{Y}_{v}\right)\right)-f(\mathcal{Y}_v)-p(\mathds{1}_{\{u\leq\mathcal{Y}_v\}}-\mathcal{Y}_v)f'(\mathcal{Y}_v)\Big)
		\end{align*}
	\end{lem}
	
	\begin{proof}
		Suppose that $\phi:\mathbb{N}\to\mathbb{N}$ defines a subsequence such that $\mathcal{Y}^{\phi(N)}\to\mathcal{Y}$ in distribution, as $N\to\infty$. 
		We define a sequence of random processes derived from the usual model with varying $\epsilon$ and $\delta$.
		More specifically, for all $\ell\geq 1$ let $\epsilon_\ell:=1/\ell^2$ and $\delta_\ell=O(\epsilon_\ell^3)$ such that \eqref{delta et epsilon} is satisfied (e.g. $1/2\ell^6$).
		Denote $\zeta_{\ell}$ the stopping time associated to the model with parameters $T,\epsilon_\ell,\delta_\ell$ and define $N_\ell$ large enough such that $\mathbb{P}(\zeta_\ell>a_{N_\ell}T)<\epsilon_\ell$ and Lemma \ref{Lemma bound step} holds, and such that $(N_\ell)_{\ell\geq 1}$ is a subsequence of $\phi$.
		Let $(\mathsf{Y}^{(\ell)}_j)_{j(2)\leq j\leq j(T-1)}$ be the process stopped at time $\zeta_\ell$, defined as previously, but with the varying parameters $\epsilon_\ell,\delta_\ell,N_\ell$.
		Since it does not depend on $\epsilon_\ell$ and $\delta_\ell$, the corresponding non-stopped continuous-time process is simply $\mathcal{Y}^{N_\ell}$.
		Note that
		\begin{align*}
			\left(\mathds{1}_{\{\zeta_\ell>a_{N_\ell}T\}}\mathsf{Y}^{(\ell)}_{j(t)-1}\right)_{t\in\intervalleff{2}{T-1}}
			=\left(\mathds{1}_{\{\zeta_\ell>a_{N_\ell}T\}}\mathcal{Y}^{N_\ell}_t\right)_{t\in\intervalleff{2}{T-1}}.
		\end{align*}
		Hence, using that $\epsilon_\ell\to 0$ as $\ell\to\infty$, it is straightforward that $(\mathsf{Y}^{(\ell)}_{j(t)-1})_{t\in\intervalleff{2}{T-1}}\to\mathcal{Y}$ in distribution as $\ell\to\infty$.
		Thanks to Skorokhod's Representation Theorem (see e.g. \cite{B99} Theorem 6.7 p.70), we can assume without loss of generality that the convergence holds almost surely as $\ell\to\infty$.
		
		We define
		\begin{align*}
			M^{(\ell)}_t:=f(\mathsf{Y}^{(\ell)}_{j(t)-1})-f(\mathsf{Y}^{(\ell)}_{j(2)-1})-\sum_{j=j(2)}^{j(t)\wedge j(\zeta_\ell)-1}\mathbb{E}\left(f(\mathsf{Y}^{(\ell)}_j)-f(\mathsf{Y}^{(\ell)}_{j-1})|\mathcal{F}^{N_\ell}_{\tau_j}\right),
		\end{align*}
		Note that $\zeta_\ell\leq  a_{N_\ell}T$ for finitely many $\ell$ almost surely, as a consequence of the choice of $\epsilon_\ell$.
		Therefore, almost surely, for all $t\in\intervalleff{2}{T-1}$, it holds that
		\begin{align*}
			\lim_{\ell\to\infty}M_t^{(\ell)}
			&=\lim_{\ell\to\infty}\left(f(\mathsf{Y}^{(\ell)}_{j(t)-1})-f(\mathsf{Y}^{(\ell)}_{j(2)-1})-\sum_{j=j(2)}^{j(t)-1}\mathbb{E}\left(f(\mathsf{Y}^{(\ell)}_{j})-f(\mathsf{Y}^{(\ell)}_{j-1})|\mathcal{F}^{N_\ell}_{\tau_j}\right)\right)\\
			&=M_t,
		\end{align*}
		thanks to Lemma \ref{Lemma bound step}.
		
		We now show that $(M^{(\ell)}_t)_{t\in\intervalleff{2}{T-1}}$ is a martingale with respect to its natural filtration, which will readily extend to its almost sure limit $(M_t)_{t\in\intervalleff{2}{T-1}}$.
		Let $2\leq t<t+r\leq T-1$ and write
		\begin{align*}
			&\mathbb{E}\left(M^{(\ell)}_{t+r}-M^{(\ell)}_t|(M^{(\ell)}_u)_{u\leq t}\right)\\
			&\hspace{1cm}=\sum_{n=j(t)}^{\infty}\mathbb{P}\left(j(t+r)\wedge j(\zeta)=n|(M^{(\ell)}_u)_{u\leq t}\right)\\
			&\hspace{1.7cm}\times\sum_{j=j(t)}^{n-1}\mathbb{E}\left( f(\mathsf{Y}^{(\ell)}_j)-f(\mathsf{Y}^{(\ell)}_{j-1})-\mathbb{E}\left(f(\mathsf{Y}^{(\ell)}_j)-f(\mathsf{Y}^{(\ell)}_{j-1})|\mathcal{F}^{N_\ell}_{\tau_j}\right)\Big|(M^{(\ell)}_u)_{u\leq t}\right).
		\end{align*}
		All the terms in the last sum are null, since the information given by $(M^{(\ell)}_u)_{u\leq t}$ is contained in $\mathcal{F}^{N_\ell}_{\tau_j}$ for all $j\geq j(t)$.
		We thus deduce that $M$ is a martingale, which entails the claim.
	\end{proof}
	
	\begin{lem}\label{Lemma sol mg pb is sol SDE}
		Any solution of the martingale problem of Lemma \ref{Lemma martingale problem} is a solution of the SDE \eqref{Bolthausen-Sznitman SDE}.
	\end{lem}
	
	Lemma \ref{Lemma sol mg pb is sol SDE} follows from Theorem 2.3 in \cite{K11}, which addresses the question of when a solution of a martingale problem is also a solution of the associated SDE, for general Markov processes.
	For a more specific treatment of this question in our setting, the reader may read Section 3.3 of \cite{BP15}, that sketches an adaptation of an elegant duality argument from \cite{BL05} (see proof of Lemma 1 therein).

	
\end{document}